\documentclass{amsart}

\usepackage{pstricks}
\usepackage{pst-node}
\usepackage{graphicx}
\usepackage{rotating}
\usepackage{amsmath,amsfonts,amssymb}

\newtheorem{theorem}{Theorem}[section]
\newtheorem{lemma}[theorem]{Lemma}
\newtheorem{proposition}[theorem]{Proposition}
\newtheorem{corollary}[theorem]{Corollary}
\newtheorem{conjecture}[theorem]{Conjecture}

\theoremstyle{definition}
\newtheorem{definition}[theorem]{Definition}

\theoremstyle{remark}
\newtheorem{remark}[theorem]{Remark}

\numberwithin{equation}{section}

\newcommand{\ipt}{\ensuremath i\!+\!2}
\newcommand{\ipo}{\ensuremath i\!+\!1}
\newcommand{\imo}{\ensuremath i\!-\!1}
\newcommand{\imt}{\ensuremath i\!-\!2}
\newcommand{\imh}{\ensuremath i\!-\!3}

\newcommand{\triple}{\ensuremath i\!-\!1,i,i\!+\!1}

\newcommand{\npo}{\ensuremath n\!+\!1}
\newcommand{\nmo}{\ensuremath n\!-\!1}
\newcommand{\nmt}{\ensuremath n\!-\!2}
\newcommand{\nmh}{\ensuremath n\!-\!3}

\newcommand{\G}{\ensuremath\mathcal{G}}
\newcommand{\C}{\ensuremath\mathcal{C}}

\newcommand{\B}{\bullet}
\newcommand{\cs}[1]{c@{\hskip #1ex}}

\newlength\cellsize 

\savebox2{%
\begin{picture}(11,11)
\put(0,0){\line(1,0){11}}
\put(0,0){\line(0,1){11}}
\put(11,0){\line(0,1){11}}
\put(0,11){\line(1,0){11}}
\end{picture}}

\newcommand\cellify[1]{\def\thearg{#1}\def\nothing{}%
\ifx\thearg\nothing
\vrule width0pt height\cellsize depth0pt\else
\hbox to 0pt{\usebox2\hss}\fi%
\vbox to 11\unitlength{
\vss
\hbox to 11\unitlength{\hss$#1$\hss}
\vss}}

\newcommand\tableau[1]{\vtop{\let\\=\cr
\setlength\baselineskip{-11000pt}
\setlength\lineskiplimit{11000pt}
\setlength\lineskip{0pt}
\halign{&\cellify{##}\cr#1\crcr}}}

\newlength\smcellsize 

\savebox3{%
\begin{picture}(10,10)
\put(0,0){\line(1,0){10}}
\put(0,0){\line(0,1){10}}
\put(10,0){\line(0,1){10}}
\put(0,10){\line(1,0){10}}
\end{picture}}

\newcommand\smcellify[1]{\def\thearg{#1}\def\nothing{}%
\ifx\thearg\nothing
\vrule width0pt height\smcellsize depth0pt\else
\hbox to 0pt{\usebox3\hss}\fi%
\vbox to 10\unitlength{
\vss
\hbox to 10\unitlength{\hss$#1$\hss}
\vss}}

\newcommand\smtableau[1]{\vtop{\let\\=\cr
\setlength\baselineskip{-10000pt}
\setlength\lineskiplimit{10000pt}
\setlength\lineskip{0pt}
\halign{&\smcellify{##}\cr#1\crcr}}}

\newcommand{\e}{\mbox{}}
\definecolor{boxgray}{gray}{.7}
\newcommand{\cb}{\color{boxgray}\rule{1\cellsize}{1\cellsize}\hspace{-\cellsize}\usebox2}

\newcommand{\stab}[3]{\begin{array}{c}\rnode{#1}{\tableau{#2}}\\\rnode{#1#1}{_{#3}}\end{array}}
\newcommand{\smstab}[3]{\begin{array}{c}\rnode{#1}{\smtableau{#2}}\\\rnode{#1#1}{_{#3}}\end{array}}

\newcommand{\sbull}[2]{\begin{array}{c}\rnode{#1}{\B}\\[-1ex]\rnode{#1#1}{\makebox[0pt]{$_{#2}$}}\end{array}}

\definecolor{lightgray}{gray}{.85}

\begin{document}


\title[Dual equivalence graphs I]{Dual equivalence graphs I: \\
  A new paradigm for Schur positivity}  

\author[S. Assaf]{Sami H. Assaf}
\address{Department of Mathematics, University of Southern California, Los Angeles, CA 90089-2532}
\email{shassaf@usc.edu}
\thanks{Work supported in part by NSF MSPRF DMS-0703567 and NSF DMS-1265728.}





\begin{abstract}
  We make a systematic study of a new combinatorial construction
  called a dual equivalence graph. We axiomatize these graphs and
  prove that their generating functions are symmetric and Schur
  positive. This provides a universal method for establishing the
  symmetry and Schur positivity of quasisymmetric functions. 
\end{abstract}

\maketitle

%
\section{Introduction}
%
\label{sec:introduction}

Symmetric function theory plays an important role in many areas of
mathematics including combinatorics, representation theory, and
algebraic geometry. Multiplicities of irreducible components,
dimensions of algebraic varieties, and various other algebraic
constructions that require the computation of certain integers may
often be translated to the computation of the Schur coefficients of a
given function. Thus a quintessential problem in symmetric functions
is to prove that a given function has nonnegative integer coefficients
when expressed as a sum of Schur functions. In this paper, we
introduce a new combinatorial construction, called a \emph{dual
  equivalence graph}, by which one can establish the symmetry and
Schur positivity of a function.

To illustrate the general problem and to demonstrate this new
solution, consider the problem of expanding the product of two Schur
functions as a sum of Schur functions, i.e.
\begin{equation}
  s_{\mu} s_{\nu} = \sum_{\lambda} c_{\mu,\nu}^{\lambda} s_{\lambda}.
  \label{e:LRC}
\end{equation}
Since Schur functions are a basis for symmetric functions, this
problem is well-posed, and since they are an integral basis, these
so-called Littlewood--Richardson coefficients $c_{\mu,\nu}^{\lambda}$
are integers. In fact, they are nonnegative. One way to see this is to
realize that the Schur functions are the characters for irreducible
representations of the general linear group, and so
$c_{\mu,\nu}^{\lambda}$ counts multiplicities of irreducible
representations in tensor products. Another way to understand the
nonnegativity is to realize Schur polynomials as Schubert classes for
the cohomology of the Grassmannian, and so $c_{\mu,\nu}^{\lambda}$
counts objects in the intersection of Schubert varieties. The
celebrated Littlewood--Richardson rule gives a direct combinatorial
interpretation for $c_{\mu,\nu}^{\lambda}$ without appealing to
representation theory or geometry. Briefly, $c_{\mu,\nu}^{\lambda}$
counts the number of standard Young tableaux of shape $\mu$ adjoined
$\nu$ that rectify to a specified standard Young tableau of shape
$\lambda$. Dual equivalence graphs abstract this rule to a general
tool with universal applicability.

The general set-up is as follows. Begin with a set $\mathcal{A}$ of
combinatorial objects together with a notion of a descent set
$\mathrm{Des}$ sending an object to a subset of positive integers. For
Littlewood--Richardson coefficients, $\mathcal{A}$ is the set of
standard Young tableaux and $\mathrm{Des}$ is the usual notion of
descents. Optionally, we may also have a nonnegative, possibly
multivariate, integer statistic associated to each object. Define the
quasisymmetric generating function for $\mathcal{A}$ with respect to
$\mathrm{Des}$ by
\begin{displaymath}
  f(X;q) \ = \ \sum_{T \in \mathcal{A}} q^{\mathrm{stat}(T)} Q_{\mathrm{Des}(T)}(X),
\end{displaymath}
where $Q$ denotes the fundamental basis for quasisymmetric functions
\cite{Gessel1984}.  

A \emph{dual equivalence} for $(\mathcal{A},\mathrm{Des})$ is a family
of involutions on $\mathcal{A}$ whose local equivalence classes are
Schur functions and which commute when their indices are far away. For
Littlewood--Richardson coefficients, these are Haiman's original dual
equivalence involutions \cite{Haiman1992}. A dual equivalence is
compatible with a statistic when the involutions preserve the
statistic.  From this framework, we obtain an explicit set
$\mathrm{Dom} \subset \mathcal{A}$ such that
\begin{displaymath}
  f(X;q) \ = \ \sum_{\lambda} \left(
    \sum_{\substack{S \in \mathrm{Dom}(\mathcal{A}) \\ \alpha(S)=\lambda}} q^{\mathrm{stat}(S)}
  \right) s_{\lambda}(X),  
\end{displaymath}
where $\alpha$ is an explicit map derived from $\mathrm{Des}$ that
associates to each element of $\mathrm{Dom}$ a partition. This is the
generalized notion of implicit rectification. For example, in the
Littlewood--Richardson case, the set $\mathrm{Dom}$ is precisely the
set of skew tableaux that rectify to a particular standard tableau of
straight shape. In particular, giving a dual equivalence for the data
$(\mathcal{A},\mathrm{Des})$ that is compatible with the statistic
proves that the generating function $f(X;q)$ is symmetric and Schur
positive and provides an explicit combinatorial formula for the Schur
coefficients.

After reviewing symmetric functions and the associated tableaux
combinatorics in Section~\ref{sec:preliminaries}, we review the dual
equivalence relation on standard tableaux. In Section~\ref{sec:deg},
we use this relation to define an edge-colored graph on standard
tableaux. We axiomatize the graph resulting in a new combinatorial
method for establishing the Schur positivity of a quasisymmetric
function. In Section~\ref{sec:jdt}, we reformulate this machinery in
terms of involutions on a set and give a more explicit
characterization of the Schur coefficients. As a first application of
this theory, Section~\ref{sec:llt} defines involutions on $k$-tuples
of tableaux that give a dual equivalence when $k=2$, giving a
surprisingly short proof of Schur positivity for the ribbon tableaux
generating function introduced by Lascoux, Leclerc, and Thibon
\cite{LLT1997} in the case of dominoes. While these involutions on
tuples of tableaux are not, in general, a dual equivalence, we
conjecture that the equivalence classes are always Schur
positive. These larger equivalence classes are studied further in a
forthcoming sequel.

\begin{center}
{\sc Acknowledgments}
\end{center}

The author is grateful to Mark Haiman for inspiring and helping to
develop many of the ideas contained in this paper and in its precursor
\cite{Assaf2007}. The author is indebted to N. Bergeron, S. Billey,
A. Garsia, M. Haiman, J. Haglund, G. Musiker, and F. Sottile for
carefully reading earlier drafts and providing feedback that greatly
improved the exposition. 

%
\section{Preliminaries}
%
\label{sec:preliminaries}

\subsection{Partitions and tableaux}
\label{sec:pre-partitions}

We represent an integer \emph{partition} $\lambda$ by the decreasing
sequence of its (nonzero) parts
$$
\lambda = (\lambda_1,\lambda_2, \ldots, \lambda_l), \;\;\;\;\;
\lambda_1 \geq \lambda_2 \geq \cdots \geq \lambda_l > 0 .
$$ We denote the size of $\lambda$ by $|\lambda| = \sum_i
\lambda_i$. If $|\lambda| = n$, we say that $\lambda$ is a
\emph{partition of $n$}. The \emph{Young diagram} of a partition
$\lambda$ is the set of points $(i,j)$ in the $\mathbb{Z}
\times\mathbb{Z}$ lattice such that $1 \leq i \leq \lambda_j$. We draw
the diagram so that each point $(i,j)$ is represented by the unit cell
southwest of the point; see Figure~\ref{fig:5441}.

\begin{figure}[ht]
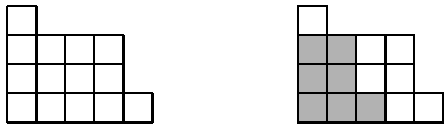

  \begin{displaymath}
    \tableau{ \e  \\
      \e & \e & \e & \e \\
      \e & \e & \e & \e \\
      \e & \e & \e & \e & \e} 
    \hspace{5\cellsize}
    \tableau{ \e  \\
      \cb & \cb & \e & \e \\
      \cb & \cb & \e & \e \\
      \cb & \cb & \cb & \e & \e} 
  \end{displaymath}
  \caption{\label{fig:5441} The Young diagram for $(5,4,4,1)$ and the
    skew diagram for $(5,4,4,1)/(3,2,2)$.}
\end{figure}

For partitions $\lambda,\mu$, we write $\mu \subset \lambda$ whenever
the diagram of $\mu$ is contained within the diagram of $\lambda$;
equivalently $\mu_i \leq \lambda_i$ for all $i$. In this case, we
define the \emph{skew diagram} $\lambda / \mu$ to be the set theoretic
difference $\lambda - \mu$, e.g. see Figure~\ref{fig:5441}. For our
purposes, we depart from the norm by \emph{not} identifying skew
shapes that are translates of one another.  


A \emph{filling} of a (skew) diagram $\lambda$ is a map $S : \lambda
\rightarrow \mathbb{Z}_+$. A \emph{semi-standard Young tableau} is a
filling that is weakly increasing along each row and strictly
increasing along each column. A semi-standard Young tableau is
\emph{standard} if it is a bijection from $\lambda$ to $[n]$, where
$[n] = \{1,2,\ldots,n\}$. For $\lambda$ a diagram of size $n$, define
\begin{eqnarray*}
  \mathrm{SSYT}(\lambda) & = & \{\mbox{semi-standard tableaux}\; 
  T : \lambda \rightarrow \mathbb{Z}_+ \}, \\
  \mathrm{SYT}(\lambda) & = & \{\mbox{standard tableaux}\; 
  T : \lambda \tilde{\rightarrow} [n]\}. 
\end{eqnarray*}
For $T \in \mathrm{SSYT}(\lambda)$, we say that $T$ has \emph{shape}
$\lambda$.  If $T$ contains entries $1^{\pi_1}, 2^{\pi_2}, \ldots$ for
some composition $\pi$, then we say $T$ has \emph{weight} $\pi$. Thus
$T \in \mathrm{SYT}(\lambda)$ if and only if $T$ has weight $(1^n)$.

\begin{figure}[ht]
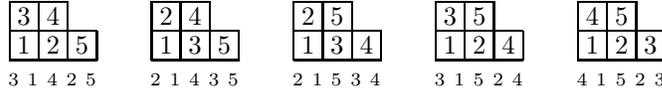

  \begin{displaymath}
    \begin{array}{ccccc}
          \stab{a}{3 & 4 \\ 1 & 2 & 5}{3 \ 1 \ 4 \ 2 \ 5} &
          \stab{b}{2 & 4 \\ 1 & 3 & 5}{2 \ 1 \ 4 \ 3 \ 5} &
          \stab{c}{2 & 5 \\ 1 & 3 & 4}{2 \ 1 \ 5 \ 3 \ 4} &
          \stab{d}{3 & 5 \\ 1 & 2 & 4}{3 \ 1 \ 5 \ 2 \ 4} &
          \stab{e}{4 & 5 \\ 1 & 2 & 3}{4 \ 1 \ 5 \ 2 \ 3}
    \end{array}
  \end{displaymath}
  \caption{\label{fig:SYT32} The standard Young tableaux of shape
    $(3,2)$ with their content reading words.}
\end{figure}

The \emph{content} of a cell of a diagram indexes the diagonal on
which it occurs, i.e. $c(x) = i-j$ when the cell $x$ lies in position
$(i,j) \in \mathbb{Z}_+ \times \mathbb{Z}_+$. The \emph{content
  reading word} of a semi-standard tableau is obtained by reading the
entries in increasing order of content, going southwest to northeast
along each diagonal (on which the content is constant). For examples,
see Figure~\ref{fig:SYT32}. 

\subsection{Symmetric functions}
\label{sec:pre-functions}

We have the familiar integral bases for $\Lambda$, the ring of
symmetric functions, from \cite{Macdonald1995}, all indexed by
partitions of $n$: the monomial symmetric functions $m_{\lambda}$, the
elementary symmetric functions $e_{\lambda}$, the complete homogeneous
symmetric functions $h_{\lambda}$, and, most importantly, the
\emph{Schur functions}, $s_{\lambda}$, which may be defined in several
ways. For the purposes of this paper, we take the tableau approach:
\begin{equation}
  s_{\lambda}(x) = \sum_{T \in \mathrm{SSYT}(\lambda)} x^{T} ,
\label{eqn:s}
\end{equation}
where $x^T$ is the monomial $x_{1}^{\pi_1} x_{2}^{\pi_2} \cdots$ when
$T$ has weight $\pi$. This formula also defines the \emph{skew Schur
  functions}, $s_{\lambda/\mu}$, by taking the sum over semi-standard
tableaux of shape $\lambda/\mu$. 

As we shall see in Section~\ref{sec:deg}, it will often be useful to
express a function in terms of Gessel's fundamental quasisymmetric
functions \cite{Gessel1984} rather than monomials. For $\sigma \in
\{\pm 1\}^{\nmo}$, the \emph{fundamental quasisymmetric function}
$Q_{\sigma}(x)$ is defined by
\begin{equation}
  Q_{\sigma}(x) = \sum_{\substack{i_1 \leq \cdots \leq i_n \\ \sigma_j = -1 \Rightarrow i_j <
      i_{j+1}}} x_{i_1} \cdots x_{i_n} .
\label{eqn:quasisym}
\end{equation}
We have indexed quasisymmetric functions by sequences of $+1$'s and
$-1$'s, though by setting $D(\sigma) = \{ i | \sigma_i = -1\}$, we may
change the indexing to the more familiar one of subsets of $[\nmo]$.

To connect quasisymmetric functions with Schur functions, for $T$ a
standard tableau on $[n]$ with content reading word $w_{T}$, define
the \emph{descent signature} $\sigma(T) \in \{\pm1\}^{\nmo}$ by
\begin{equation}
  \sigma(T)_{i} \; = \; \left\{ 
    \begin{array}{ll}
      +1 & \; \mbox{if $i$ appears to the left of $\ipo$ in $w_T$} \\
      -1 & \; \mbox{if $\ipo$ appears to the left of $i$ in $w_T$}
    \end{array} \right. .
\label{eqn:sigma}
\end{equation}
For example, the descent signatures for the tableaux in
Figure~\ref{fig:SYT32} are $+-++, \ -+-+, \ -++-, \ +-+-, \ ++-+$,
from left to right.  Note that if we replace the content reading word
with either the row or column reading word, the signature given by
\eqref{eqn:sigma} remains unchanged.

\begin{proposition}[\cite{Gessel1984}]
  The Schur function $s_{\lambda}$ is expressed in terms of
  quasisymmetric functions by
  \begin{equation}
    s_{\lambda}(x) = \sum_{T \in \mathrm{SYT}(\lambda)} Q_{\sigma(T)}(x) .
  \label{eqn:quasi-s}
  \end{equation}
\label{prop:quasisym}
\end{proposition}

Comparing \eqref{eqn:s} with \eqref{eqn:quasi-s}, using
quasisymmetric functions instead of monomials allows us to work with
standard tableaux rather than semi-standard tableaux. One advantage of
this formula is that unlike \eqref{eqn:s}, the right hand side of
\eqref{eqn:quasi-s} is finite. Continuing with the example in
Figure~\ref{fig:SYT32},
\begin{displaymath}
  s_{(3,2)}(x) = Q_{+-++}(x) + Q_{-+-+}(x) + Q_{-++-}(x) + Q_{+-+-}(x)
  + Q_{++-+}(x).
\end{displaymath}

\subsection{Dual equivalence}
\label{sec:deg-standard}

Dual equivalence was first explicitly defined by Haiman
\cite{Haiman1992} as a relation on tableaux dual to \emph{jeu de
  taquin} equivalence under the Schensted correspondence. The
elementary moves defined below are sometimes called dual Knuth moves,
and can be obtained by taking inverses of permutations that are
elementary Knuth equivalent.

\begin{figure}[ht]
    \begin{displaymath}
      \begin{array}{ccc}
        \{ 2314 \stackrel{d_2}{\longleftrightarrow} 1324
        \stackrel{d_3}{\longleftrightarrow} 1423 \} &
        \{ 2143 \begin{array}{c}
          \stackrel{d_2}{\longleftrightarrow} \\[-1ex]
          \stackrel{\displaystyle\longleftrightarrow}{_{d_3}}
        \end{array} 3142 \} &
        \{ 1432 \stackrel{d_2}{\longleftrightarrow} 2431
        \stackrel{d_3}{\longleftrightarrow} 3421 \} \\
        \{ 2341 \stackrel{d_2}{\longleftrightarrow} 1342 
        \stackrel{d_3}{\longleftrightarrow} 1243 \} & &
        \{ 4312 \stackrel{d_2}{\longleftrightarrow} 4213
        \stackrel{d_3}{\longleftrightarrow} 3214 \} \\
        \{ 2134 \stackrel{d_2}{\longleftrightarrow} 3124
        \stackrel{d_3}{\longleftrightarrow} 4123 \} &
        \{ 2413 \begin{array}{c}
          \stackrel{d_2}{\longleftrightarrow} \\[-1ex]
          \stackrel{\displaystyle\longleftrightarrow}{_{d_3}}
        \end{array} 3412 \} &
        \{ 4132 \stackrel{d_2}{\longleftrightarrow} 4231
        \stackrel{d_3}{\longleftrightarrow} 3241 \} 
      \end{array} 
    \end{displaymath}
    \caption{\label{fig:elementary}The nontrivial dual equivalence
      classes of $\mathfrak{S}_4$.}
\end{figure}

\begin{definition}[\cite{Haiman1992}]
  Define the \emph{elementary dual equivalence involution} $d_i$,
  $1<i<n$, on permutations $w$ as follows. If $i$ lies between $i-1$
  and $i+1$ in $w$, then $d_i(w)=w$. Otherwise, $d_i$ interchanges $i$
  and whichever of $i\pm 1$ is further away from $i$. Two permutations
  $w$ and $u$ are \emph{dual equivalent} if there exists a sequence
  $i_1,\ldots,i_k$ such that $u = d_{i_k} \cdots d_{i_1}(w)$.
  \label{defn:ede}
\end{definition}

For examples, see Figure~\ref{fig:elementary}. Two standard tableaux
of the same shape are \emph{dual equivalent} if their content reading
words are; for examples, see Figure~\ref{fig:classes}.

\begin{figure}[ht]
    \begin{displaymath}
      \left\{ \
      \tableau{2 \\ 1 & 3 & 4} 
      \stackrel{d_2}{\longleftrightarrow} 
      \tableau{3 \\ 1 & 2 & 4}
      \stackrel{d_3}{\longleftrightarrow} 
      \tableau{4 \\ 1 & 2 & 3} 
      \ \right\}
      \hspace{1.5em}
      \left\{ \
      \tableau{2 & 4 \\ 1 & 3} 
      \begin{array}{c}
        \stackrel{d_2}{\longleftrightarrow} \\[-.5ex]
        \stackrel{\displaystyle\longleftrightarrow}{_{d_3}}
      \end{array}
      \tableau{3 & 4 \\ 1 & 2}
      \ \right\}
      \hspace{1.5em}
      \left\{ \
      \tableau{4 \\ 3 \\ 1 & 2}
      \stackrel{d_2}{\longleftrightarrow} 
      \tableau{4 \\ 2 \\ 1 & 3}
      \stackrel{d_3}{\longleftrightarrow} 
      \tableau{3 \\ 2 \\ 1 & 4} 
      \ \right\}
    \end{displaymath}
    \caption{\label{fig:classes}The nontrivial dual equivalence
      classes of $\mathrm{SYT}$ of size $4$.}
\end{figure}

\begin{proposition}[\cite{Haiman1992}]
  Two standard tableaux on partition shapes are dual equivalent if and
  only if they have the same shape.
\label{prop:deshape}
\end{proposition}

Propositions~\ref{prop:quasisym} and \ref{prop:deshape} together allow
us to express a Schur function in terms of dual equivalence classes as
\begin{equation}
  s_{\lambda}(X) \ = \ \sum_{T \in [T_{\lambda}]} Q_{\sigma(T)}(X),
\label{e:classes}
\end{equation}
where $[T_{\lambda}]$ denotes the dual equivalence class of some (any)
fixed standard tableau $T_{\lambda}$ of shape $\lambda$. This paradigm
shift to summing over objects in a dual equivalence class is the main
idea underlying dual equivalence graphs presented below.

Dual equivalence also applies to skew tableaux, and the involutions
$d_i$ commute with rectification via jeu de taquin
\cite{Haiman1992}. For example, the skew tableaux in
Figure~\ref{fig:jdt} rectify to the tableaux in
Figure~\ref{fig:classes} respectively. Taking generating functions, we
see that the dual equivalence classes correspond precisely to the
Schur expansion
\begin{displaymath}
  s_{(1)}s_{(2,1)} \ = \ s_{(3,1)} + s_{(2,2)} + s_{(2,1,1)}.
\end{displaymath}

\begin{figure}[ht]
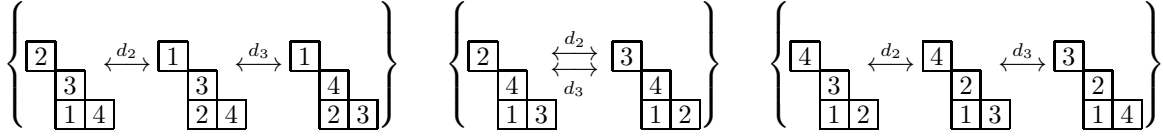

    \begin{displaymath}
      \left\{
      \tableau{2 \\ & 3 \\ & 1 & 4} 
      \hspace{-.7em} \stackrel{d_2}{\longleftrightarrow}  
      \tableau{1 \\ & 3 \\ & 2 & 4} 
      \hspace{-.7em} \stackrel{d_3}{\longleftrightarrow}  
      \tableau{1 \\ & 4 \\ & 2 & 3} 
      \right\} 
      \hspace{1.5em}
      \left\{
      \tableau{2 \\ & 4 \\ & 1 & 3} 
      \hspace{-.7em} 
      \begin{array}{c}
        \stackrel{d_2}{\longleftrightarrow} \\[-.5ex]
        \stackrel{\displaystyle\longleftrightarrow}{_{d_3}}
      \end{array}
      \tableau{3 \\ & 4 \\ & 1 & 2} 
      \right\} 
      \hspace{1.5em}
      \left\{
      \tableau{4 \\ & 3 \\ & 1 & 2} 
      \hspace{-.7em} \stackrel{d_2}{\longleftrightarrow}  
      \tableau{4 \\ & 2 \\ & 1 & 3} 
      \hspace{-.7em} \stackrel{d_3}{\longleftrightarrow}  
      \tableau{3 \\ & 2 \\ & 1 & 4} 
      \right\}
    \end{displaymath}
    \caption{\label{fig:jdt}Dual equivalence classes for $s_{(1)}s_{(2,1)}$.}
\end{figure}

%
\section{Dual equivalence graphs}
%
\label{sec:deg}

\subsection{Axiomatization of dual equivalence}
\label{sec:deg-general}

Construct a graph whose edges are colored on standard tableaux of
partition shape from the dual equivalence relation in the following
way.  Whenever two standard tableaux $T,U$ have content reading words
that differ by an elementary dual equivalence for $\triple$, connect
$T$ and $U$ with an edge colored by $i$. Associate to each tableau $T$
the signature $\sigma(T)$ defined by \eqref{eqn:sigma}. For example,
see Figure~\ref{fig:G5}. Several more examples are given in
Appendix~\ref{app:DEGs}.

\begin{figure}[ht]
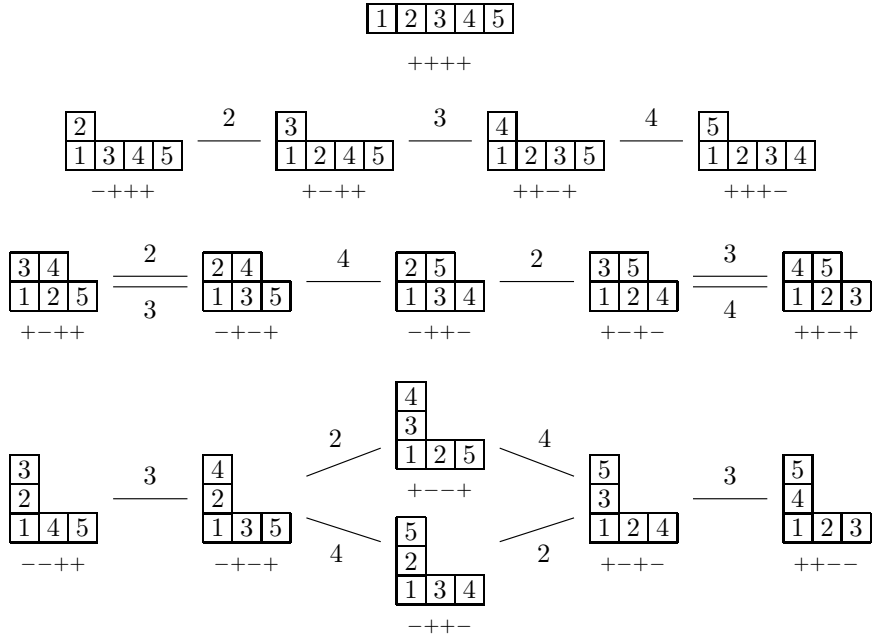

  \begin{center}
    \begin{displaymath}
      \begin{array}{c}
        \begin{array}{c}
          \stab{h}{1 & 2 & 3 & 4 & 5}{++++}
        \end{array} \\[5ex] 
        \begin{array}{\cs{6} \cs{6} \cs{6} c}
          \stab{h}{2 \\ 1 & 3 & 4 & 5}{-+++} &
          \stab{i}{3 \\ 1 & 2 & 4 & 5}{+-++} &
          \stab{j}{4 \\ 1 & 2 & 3 & 5}{++-+} &
          \stab{k}{5 \\ 1 & 2 & 3 & 4}{+++-}
        \end{array} \\[7ex]
        \begin{array}{\cs{7} \cs{7} \cs{7} \cs{7} c}
          \stab{a}{3 & 4 \\ 1 & 2 & 5}{+-++} &
          \stab{b}{2 & 4 \\ 1 & 3 & 5}{-+-+} &
          \stab{c}{2 & 5 \\ 1 & 3 & 4}{-++-} &
          \stab{d}{3 & 5 \\ 1 & 2 & 4}{+-+-} &
          \stab{e}{4 & 5 \\ 1 & 2 & 3}{++-+}
        \end{array} \\[6ex]
        \begin{array}{\cs{7} \cs{7} \cs{7} \cs{7} c}
          & & \stab{w}{4 \\ 3 \\ 1 & 2 & 5}{+--+} & & \\[-4ex]
              \stab{u}{3 \\ 2 \\ 1 & 4 & 5}{--++} &
              \stab{v}{4 \\ 2 \\ 1 & 3 & 5}{-+-+} & &
              \stab{y}{5 \\ 3 \\ 1 & 2 & 4}{+-+-} &
              \stab{z}{5 \\ 4 \\ 1 & 2 & 3}{++--} \\[-5ex]
          & & \stab{x}{5 \\ 2 \\ 1 & 3 & 4}{-++-} & &
        \end{array}
        \psset{nodesep=6pt,linewidth=.1ex}
        \ncline  {h}{i} \naput{2}
        \ncline  {i}{j} \naput{3}
        \ncline  {j}{k} \naput{4}
        \ncline[offset=2pt] {a}{b} \naput{2}
        \ncline[offset=2pt] {b}{a} \naput{3}
        \ncline             {b}{c} \naput{4}
        \ncline             {c}{d} \naput{2}
        \ncline[offset=2pt] {d}{e} \naput{3}
        \ncline[offset=2pt] {e}{d} \naput{4}
        \ncline {u}{v}  \naput{3}
        \ncline {v}{w}  \naput{2}
        \ncline {v}{x}  \nbput{4}
        \ncline {w}{y}  \naput{4}
        \ncline {x}{y}  \nbput{2}
        \ncline {y}{z}  \naput{3}
      \end{array}
    \end{displaymath}
    \caption{\label{fig:G5}The standard dual equivalence graphs
      $\G_{5}, \G_{4,1}, \G_{3,2}$ and $\G_{3,1,1}$.}
  \end{center}
\end{figure}

The connected components of the graph so constructed are the dual
equivalence classes of standard tableaux. Let $\G_{\lambda}$ denote
the subgraph on tableaux of shape $\lambda$.
Proposition~\ref{prop:deshape} states that the $\G_{\lambda}$ exactly
give the connected components of the graph.

Define the generating function associated to $\G_{\lambda}$ by
\begin{equation}
  \sum_{v \in V(\G_{\lambda})} Q_{\sigma(v)}(x) = s_{\lambda}(x).
\label{eqn:glamschur}
\end{equation}
By Proposition~\ref{prop:quasisym}, this is Gessel's quasisymmetric
function expansion for a Schur function. In particular, the generating
function of any graph who vertices have signatures and whose connected
components are isomorphic to the graphs $\G_{\lambda}$ is Schur
positive.

In this section, we characterize $\G_{\lambda}$ in terms of edges and
signatures so that we can readily identify those graphs that are
isomorphic to some $\G_{\lambda}$.

\begin{definition}
  A \emph{signed, colored graph of type $(n,N)$} consists of the
  following data:
  \begin{itemize}
    \item a finite vertex set $V$;
    \item a signature function $\sigma : V \rightarrow \{\pm
      1\}^{N-1}$;
    \item for each $1 < i < n$, a collection $E_i$ of pairs of
      distinct vertices of $V$.
  \end{itemize}
  We denote such a graph by $\G = (V,\sigma,E_{2} \cup\cdots\cup E_{\nmo})$ or
  simply $(V,\sigma,E)$.
\label{defn:scg}
\end{definition}

\begin{definition}
  A signed, colored graph $\G = (V,\sigma,E)$ of type $(n,N)$ is a
  \emph{dual equivalence graph of type $(n,N)$} if $n \leq N$ and the
  following hold:
  \begin{itemize}
    
  \item[(ax$1$)] For $w \in V$ and $1<i<n$, $\sigma(w)_{\imo} =
    -\sigma(w)_{i}$ if and only if there exists $x \in V$ such that
    $\{w,x\} \in E_{i}$. Moreover, $x$ is unique when it exists.

  \item[(ax$2$)] For $\{w,x\} \in E_{i}$, $\sigma(w)_j = -\sigma(x)_j$
    for $j=\imo,i$, and $\sigma(w)_h = \hspace{1ex}\sigma(x)_h$ for $h
    < \imt$ and $h > \ipo$.
      
  \item[(ax$3$)] For $\{w,x\} \in E_{i}$, if $\sigma(w)_{\imt} =
    -\sigma(x)_{\imt}$, then $\sigma(w)_{\imt} = -\sigma(w)_{\imo}$,
    and if $\sigma(w)_{\ipo} = -\sigma(x)_{\ipo}$, then
    $\sigma(w)_{\ipo} = -\sigma(w)_{i}$.
    
  \item[(ax$4$)] Every connected component of $(V,\sigma, E_{\imo}
    \cup E_{i})$ appears in Figure~\ref{fig:lambda4} and every
    connected component of $(V,\sigma,E_{\imt} \cup E_{\imo} \cup
    E_{i})$ appears in Figure~\ref{fig:lambda5}.

  \item[(ax$5$)] If $\{w,x\} \in E_i$ and $\{x,y\} \in E_j$ for $|i-j|
    \geq 3$, then $\{w,v\} \in E_j$ and $\{v,y\} \in E_i$ for some $v
    \in V$.

  \item[(ax$6$)] Any two vertices of a connected component of
    $(V,\sigma,E_2 \cup\cdots\cup E_i)$ may be connected by a path
    crossing at most one $E_i$ edge.

  \end{itemize}
\label{defn:deg}
\end{definition}

Note that if $n>4$, then the allowed structure for connected
components of $(V,\sigma,E_{\imt} \cup E_{\imo} \cup E_{i})$ dictates
that every connected component of $(V,\sigma, E_{\imo} \cup E_{i})$
appears in Figure~\ref{fig:lambda4}.

\begin{figure}[ht]
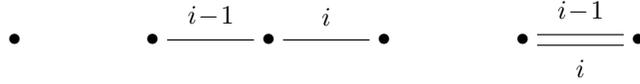

  \begin{displaymath}
    \begin{array}{\cs{11}\cs{9}\cs{9}\cs{11}\cs{9}c}
      \B & \rnode{a}{\B} & \rnode{b}{\B} & \rnode{c}{\B} & \rnode{d}{\B} & \rnode{e}{\B}
    \end{array}
    \psset{nodesep=3pt,linewidth=.1ex}
    \ncline            {a}{b} \naput{\imo}
    \ncline            {b}{c} \naput{i}
    \ncline[offset=2pt]{d}{e} \naput{\imo}
    \ncline[offset=2pt]{e}{d} \naput{i}
  \end{displaymath}
  \caption{\label{fig:lambda4} Allowed $2$-color connected
    components of a dual equivalence graph.}
\end{figure}

\begin{figure}[ht]
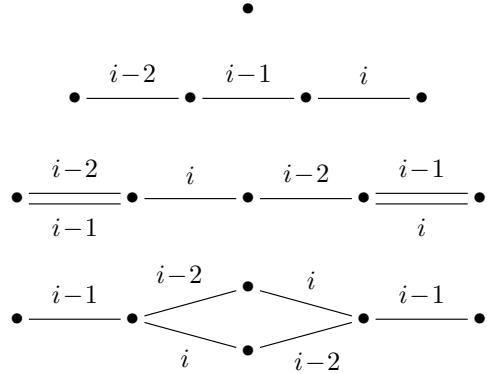

  \begin{displaymath}
    \begin{array}{c}
        \begin{array}{c}
          \B
        \end{array}\\[5ex]
        \begin{array}{\cs{9} \cs{9} \cs{9} c}
          \rnode{h}{\B} & \rnode{i}{\B} & \rnode{j}{\B} & \rnode{k}{\B}
        \end{array} \\[6ex]
        \begin{array}{\cs{9} \cs{9} \cs{9} \cs{9} c}
          \rnode{a}{\B} & \rnode{b}{\B} & \rnode{c}{\B} & \rnode{d}{\B} & \rnode{e}{\B}
        \end{array} \\[5ex]
        \begin{array}{\cs{9} \cs{9} \cs{9} \cs{9} c}
          & & \rnode{w}{\B} & & \\
          \rnode{u}{\B} & \rnode{v}{\B} & & \rnode{y}{\B} & \rnode{z}{\B} \\
          & & \rnode{x}{\B} & &
        \end{array}
        \psset{nodesep=2pt,linewidth=.1ex}
        \ncline  {h}{i} \naput{\imt}
        \ncline  {i}{j} \naput{\imo}
        \ncline  {j}{k} \naput{i}
        \ncline[offset=2pt] {a}{b} \naput{\imt}
        \ncline[offset=2pt] {b}{a} \naput{\imo}
        \ncline             {b}{c} \naput{i}
        \ncline             {c}{d} \naput{\imt}
        \ncline[offset=2pt] {d}{e} \naput{\imo}
        \ncline[offset=2pt] {e}{d} \naput{i}
        \ncline {u}{v}  \naput{\imo}
        \ncline {v}{w}  \naput{\imt}
        \ncline {v}{x}  \nbput{i}
        \ncline {w}{y}  \naput{i}
        \ncline {x}{y}  \nbput{\imt}
        \ncline {y}{z}  \naput{\imo}
      \end{array}
    \end{displaymath}
  \caption{\label{fig:lambda5} Allowed $3$-color connected
    components of a dual equivalence graph.}
\end{figure}

Every connected component of a dual equivalence graph of type $(n,N)$
is again a dual equivalence graph of type $(n,N)$. 

It is often useful to consider a restricted set of edges of a signed,
colored graph. To be precise, for $m \leq n$ and $M \leq N$, the
\emph{$(m,M)$-restriction} of a signed, colored graph $\G$ of type
$(n,N)$ consists of the vertex set $V$, signature function $\sigma: V
\rightarrow \{\pm 1\}^{M-1}$ obtained by truncating $\sigma$ at $M-1$,
and the edge set $E_{2} \cup\cdots\cup E_{m-1}$. For $m \leq n,M \leq
N$, the $(m,M)$-restriction of a dual equivalence graph of type
$(n,N)$ is a dual equivalence graph of type $(m,M)$.

The graph for $\G_{\lambda'}$ is obtained from $\G_{\lambda}$ by
conjugating each standard tableau and multiplying the signatures
coordinate-wise by $-1$. Therefore the structure of $\G_{(2,1,1,1)},
\G_{(2,2,1)}$ and $\G_{(1,1,1,1,1)}$ is also indicated by
Figure~\ref{fig:G5}. Comparing this with Figure~\ref{fig:lambda5},
axiom $4$ stipulates that the components of a dual equivalence graph
restricted to three consecutive edge colors are exactly the graphs for
$\G_{\lambda}$ when $\lambda$ is a partition of $5$.

\begin{proposition}
  For $\lambda$ a partition of $n$, $\G_{\lambda}$ is a dual
  equivalence graph of type $(n,n)$.
\label{prop:good-defn}
\end{proposition}

\begin{proof}
  For $T \in \mathrm{SYT}(\lambda)$, $\sigma(T)_{\imo} = -\sigma(T)_i$
  if and only if $i$ does not lie between $\imo$ and $\ipo$ in the
  content reading word of $T$. In this case, there exists $U \in
  \mathrm{SYT}(\lambda)$ such that $T$ and $U$ differ by an elementary
  dual equivalence for $\triple$. Therefore $U$ is obtained from $T$
  by swapping $i$ with $\imo$ or $\ipo$, whichever lies further away,
  with the result that $\sigma(T)_j = -\sigma(U)_j$ for $j=\imo,i$ and
  also $\sigma(T)_h = \sigma(U)_h$ for $h<\imt$ and $\ipo<h$. This
  verifies axioms $1$ and $2$.
  
  For axiom $3$, if $\sigma(T)_{\imt} = -\sigma(U)_{\imt}$, then $i$
  and $\imo$ have interchanged positions with $\imt$ lying between, so
  that $T$ and $U$ also differ by an elementary dual equivalence for
  $i\!-\!2,i\!-\!1,i$, and similarly for $\ipo$. From this, we obtain
  an explicit description of double edges as those connecting vertices
  where $i-2$ and $i+1$ lie between $i-1$ and $i$. Therefore axiom $4$
  becomes a straightforward, finite check on permutations of $5$. If
  $|i-j| \geq 3$, then $\{\triple\} \cap \{j\!-\!1,j,j\!+\!1\} =
  \emptyset$, so the elementary dual equivalences for $\triple$ and
  for $j\!-\!1,j,j\!+\!1$ commute, thereby demonstrating axiom $5$.

  Finally, for $T,U \in \mathrm{SYT}(\lambda)$, $|\lambda| = \ipo$, we
  must show that there exists a path from $T$ to $U$ crossing at most
  one $E_i$ edge. Let $\C_T$ (resp. $\C_U$) denote the connected
  component of the $(i,i)$-restriction of $\G_{\lambda}$ containing
  $T$ (resp.  $U$).  Let $\mu$ (resp. $\nu$) be the shape of $T$
  (resp. $U$) with the cell containing $\ipo$ removed. Then $\C_T
  \cong \G_{\mu}$ and $\C_U \cong \G_{\nu}$. If $\mu = \nu$, then, by
  Proposition~\ref{prop:deshape}, $\C_T = \C_U$ and axiom $6$
  holds. Assume, then, that $\mu \neq \nu$. Since $\mu,\nu \subset
  \lambda$ and $|\mu| = |\nu| = |\lambda|-1$, both cells $\lambda/\mu$
  and $\lambda/\nu$ must be northeastern corners of
  $\lambda$. Therefore there exists $T' \in \mathrm{SYT}(\lambda)$
  with $i$ in position $\lambda/\nu$, $\ipo$ in position
  $\lambda/\mu$, and $\imo$ between $i$ and $\ipo$ in the content
  reading word of $T'$.  Let $U'$ be the result of swapping $i$ and
  $\ipo$ in $T'$, in particular, $\{T',U'\} \in E_{i}$. By
  Proposition~\ref{prop:deshape}, $T'$ is in $\C_T$ and $U'$ is in
  $\C_U$, hence there exists a path from $T$ to $T'$ and a path from
  $U'$ to $U$ each crossing only edges $E_h$, $h < i$. This
  establishes axiom $6$.
\end{proof}

\begin{definition}
  For partitions $\lambda \subset \rho$, with $|\lambda| = n$ and
  $|\rho| = N$, choose a tableau $A$ of shape $\rho/\lambda$ with
  entries $n+1,\ldots,N$. Define the set of standard Young tableaux of
  shape $\lambda$ augmented by $A$, denoted
  $\mathrm{ASYT}(\lambda,A)$, to be those $T \in \mathrm{SYT}(\rho)$
  such that $T$ restricted to $\rho/\lambda$ is $A$. Let
  $\G_{\lambda,A}$ be the signed, colored graph of type $(n,N)$
  constructed on $\mathrm{ASYT}(\lambda,A)$ with $i$-edges given by
  elementary dual equivalences for $\triple$ with $i<n$.
  \label{defn:augment}
\end{definition}

Note that $\G_{\lambda,A}$ is a dual equivalence graph of type
$(n,N)$, and the $(n,n)$-restriction of $\G_{\lambda,A}$ is
$\G_{\lambda}$.

Proposition~\ref{prop:good-defn} is the first step towards justifying
Definition~\ref{defn:deg}, and also allows us to refer to
$\G_{\lambda}$ as the \emph{standard dual equivalence graph
  corresponding to $\lambda$}. In order to show the converse, when two
graphs satisfy axiom 1, as all graphs in this paper do, we define an
isomorphism between them to be a sign-preserving bijection on vertex
sets that respects color-adjacency. Precisely, we have the following.

\begin{definition}
  A \emph{morphism} between two signed, colored graphs of type $(n,N)$
  satisfying dual equivalence graph axiom $1$, say $\G=(V,\sigma,E)$
  and $\mathcal{H}=(W,\tau,F)$, is a map $\phi : V \rightarrow W$ such
  that for every $u,v \in V$
  \begin{itemize}
  \item for every $1 \leq i < N$, we have $\sigma(v)_i =
    \tau(\phi(v))_i$, and 
  \item for every $1 < i < n$, if $\{u,v\} \in E_i$, then
    $\{\phi(u),\phi(v)\} \in F_i$.
  \end{itemize}
  A morphism is an \emph{isomorphism} if it is a bijection on vertex
  sets.
\label{defn:isomorphism}
\end{definition}

\begin{lemma}
  If $\phi$ is a morphism from a signed, colored graph $\G$ of type
  $(n,N)$ satisfying axiom $1$ to an augmented standard dual
  equivalence graph $\G_{\lambda,A}$, then $\phi$ is surjective.
  \label{lem:onto}
\end{lemma}

\begin{proof}
  Suppose $T = \phi(v)$ for some $T \in \mathrm{ASYT}(\lambda,A)$ and
  some vertex $v$ of $\G$. Then for every $1 < i < n$, if $\{T,U\} \in
  E_{i}$, then since $\sigma(v) = \sigma(T)$, by axiom $1$ there
  exists a unique vertex $w$ of $\G$ such that $\{v,w\} \in E_i$ in
  $\G$. Since $\phi$ is a morphism, we must have $\{T,\phi(w)\} \in
  E_i$ in $\G_{\lambda,A}$. Thus by the uniqueness condition of axiom
  $1$, $\phi(w) = U$, and so $U$ also lies in the image of
  $\phi$. Therefore the $i$-neighbor of any vertex in the image of
  $\phi$ also lies in the image since $\phi$ preserves
  $i$-edges. Since $\G_{\lambda,A}$ is connected, $\phi$ is
  surjective.
\end{proof}

The final justification of this axiomatization is the following
converse of Proposition~\ref{prop:good-defn}.

\begin{theorem}
  Every connected component of a dual equivalence graph of type
  $(n,n)$ is isomorphic to $\G_{\lambda}$ for a unique partition
  $\lambda$ of $n$.
\label{thm:isomorphic}
\end{theorem}

The proof of Theorem~\ref{thm:isomorphic} is the content of
Section~\ref{sec:deg-proof}.  We conclude this section by interpreting
Theorem~\ref{thm:isomorphic} in terms of symmetric functions.

\begin{corollary}
  Let $\G$ be a dual equivalence graph of type $(n,n)$ such that every
  vertex is assigned some additional statistic $\alpha$ that is
  constant on connected components of $\G$. Then
  \begin{equation}
    \sum_{v \in V(\G)} q^{\alpha(v)} Q_{\sigma(v)}(X) = 
    \sum_{\lambda} \sum_{\C \cong \G_{\lambda}} q^{\alpha(\C)} s_{\lambda}(X) .
    \label{eqn:schurpos}
  \end{equation}
  where the inner sum is over connected components of $\G$ that are
  isomorphic to $\G_{\lambda}$. In particular, the generating function
  for $\G$ so defined is symmetric and Schur positive.
\label{cor:schurpos}
\end{corollary}

\subsection{The structure of dual equivalence graphs}
\label{sec:deg-proof}

We begin the proof of Theorem~\ref{thm:isomorphic} by showing that the
standard dual equivalence graphs are non-redundant in the sense that
they are mutually non-isomorphic and have no nontrivial
automorphisms. Both results stem from the observation that
$\G_{\lambda}$ contains a unique vertex such that the composition
formed by the lengths of the runs of $+1$'s in the signature gives a
maximal partition.

\begin{proposition} 
  If $\phi: \G_{\lambda} \rightarrow \G_{\mu}$ is an isomorphism, then
  $\lambda = \mu$ and $\phi=\mathrm{id}$.
\label{prop:noauto-noniso}
\end{proposition}

\begin{proof}
  Let $T_{\lambda}$ be the tableau obtained by filling the numbers 1
  through $n$ into the rows of $\lambda$ from left to right, bottom to
  top, in which case $\sigma(T_{\lambda}) = +^{\lambda_1\!-\!1}, -,
  +^{\lambda_2\!-\!1}, -, \cdots$. For any standard tableau $T$ such
  that $\sigma(T) = \sigma(T_{\lambda})$, the numbers $1$ through
  $\lambda_1$, and also $\lambda_1 + 1$ through $\lambda_1 +
  \lambda_2$, and so on, must form horizontal strips. In particular,
  if $\sigma(T) = \sigma(T_{\lambda})$ for some $T$ of shape $\mu$,
  then $\lambda \leq \mu$ with equality if and only if $T =
  T_{\lambda}$.

  Suppose $\phi : \G_{\lambda} \rightarrow \G_{\mu}$ is an
  isomorphism. Let $T_{\lambda}$ be as above for $\lambda$, and let
  $T_{\mu}$ be the corresponding tableau for $\mu$. Then since
  $\sigma(\phi(T_{\lambda})) = \sigma(T_{\lambda})$, $\lambda \leq
  \mu$. Conversely, since $\sigma(\phi^{-1}(T_{\mu})) =
  \sigma(T_{\mu})$, $\mu \leq \lambda$.  Therefore $\lambda=\mu$.
  Furthermore, $\phi(T_{\lambda}) = T_{\lambda}$. For $T \in
  \mathrm{SYT}(\lambda)$ such that $\{T_{\lambda},T\} \in E_i$, we
  have $\{T_{\lambda},\phi(T)\} \in E_i$, so $\phi(T) = T$ by dual
  equivalence axiom $1$.  Extending this, every tableau connected to a
  fixed point by some sequence of edges is also a fixed point for
  $\phi$, hence $\phi = \mathrm{id}$ on each $\G_{\lambda}$ by
  Proposition~\ref{prop:deshape}.
\end{proof}

In order to avoid cumbersome notation, as we investigate the
connection between an arbitrary dual equivalence graph and the
standard one, we will often abuse notation by simultaneously referring
to $\sigma$ and $E$ as the signature function and edge set for both
graphs.

\begin{definition}
  Let $\G = (V,\sigma,E)$ be a signed, colored graph of type $(n,N)$
  satisfying axiom $1$. For $1 < i < N$, we say that a vertex $w \in
  V$ \emph{admits an $i$-neighbor} if $\sigma(w)_{\imo} =
  -\sigma(w)_i$.
\label{defn:neighbor}
\end{definition}

For $1 < i < n$, if $\sigma(w)_{\imo} = -\sigma(w)_i$ for some $w \in
V$, then axiom $1$ implies the existence of $x \in V$ such that
$\{w,x\} \in E_i$. That is, if $w$ admits an $i$-neighbor for some
$1<i<n$, then $w$ has an $i$-neighbor in $\G$. For $n \leq i < N$,
though $i$-edges do not exist in $\G$, if $\G$ were the restriction of
a graph of type $(\ipo,N)$ also satisfying axiom $1$, then the
condition $\sigma(w)_{\imo} = -\sigma(w)_i$ would imply the existence
of a vertex $x$ such that $\{w,x\} \in E_{i}$ in the type $(\ipo,N)$
graph. When convenient, $E_{i}$ may be regarded as an involution,
where $E_i(w) = x$ if $w$ admits an $i$-neighbor and $\{w,x\}
\in E_{i}$, and otherwise $E_i(w) = w$.

Recall the notion of augmenting a partition $\lambda$ by a skew
tableau $A$ and the resulting dual equivalence graph $\G_{\lambda,A}$
from Definition~\ref{defn:augment}.

\begin{lemma}
  Let $\G = (V,\sigma,E)$ be a connected dual equivalence graph of
  type $(n,N)$ with $n<N$, and let $\phi$ be an isomorphism from the
  $(n,n)$-restriction of $\G$ to $\G_{\lambda}$ for some partition
  $\lambda$ of $n$.  Then there exists a semi-standard tableau $A$ of
  shape $\rho/\lambda$, $|\rho| = N$, with entries $n+1,\ldots,N$ such
  that $\phi$ gives an isomorphism from $\G$ to $\G_{\lambda,A}$.
  Moreover, the position of the cell of $A$ containing $n+1$ is
  unique.
\label{lem:extend-signs}
\end{lemma}

\begin{proof}
  By axiom $2$ and the fact that $\G$ is connected, $\sigma_h$ is
  constant on $\G$ for $h \geq \npo$. Therefore once a suitable cell
  for $\npo$ has been chosen, the cells for $n+2, \cdots, N$ may be
  chosen in any way that gives the correct signature. One solution is
  to place $j$ north of the first column if $\sigma_{j-1} = -1$ or
  east of the first row if $\sigma_{j-1} = +1$ for
  $j=n+2,\cdots,N$. Assume, then, that $N = \npo$.
  
  By dual equivalence axiom $2$, $\sigma_n$ is constant on connected
  components of the $(\nmo,\npo)$-restriction of $\G$. By
  Proposition~\ref{prop:deshape}, a connected component of the
  $(\nmo,\nmo)$-restriction of $\G_{\lambda}$ consists of all standard
  Young tableaux where $n$ lies in a particular northeastern cell of
  $\lambda$. Therefore, for each connected component of the
  $(\nmo,\npo)$-restriction of $\G$, we may identify its image under
  $\phi$ with $\G_{\mu}$ for some partition $\mu \subset \lambda$,
  $|\mu| = \nmo$, with $n$ lying in position $\lambda/\mu$. We will
  show that $\sigma_n$ has the monotonicity property on connected
  components of the $(\nmo,\npo)$-restriction of $\G$ depicted in
  Figure~\ref{fig:monotonic}, i.e., there is an inner corner above
  which $\sigma_n = +1$ and below which $\sigma_n = -1$.

  \begin{figure}[ht]
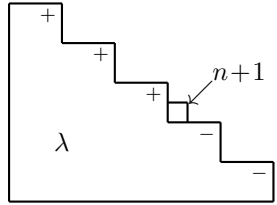

    \begin{center}
      \psset{xunit=2em}
      \psset{yunit=1.5em}
      \pspicture(0,0)(5,5)
      \psline(0,0)(0,5)
      \psline(1,4)(1,5)
      \psline(2,3)(2,4)
      \psline(3,2)(3,3)
      \psline(4,1)(4,2)
      \psline(5,0)(5,1)
      \psline(0,0)(5,0)
      \psline(4,1)(5,1)
      \psline(3,2)(4,2)
      \psline(2,3)(3,3)
      \psline(1,4)(2,4)
      \psline(0,5)(1,5)
      \rput(0.75,4.7){$_+$}
      \rput(1.75,3.7){$_+$}
      \rput(2.75,2.7){$_+$}
      \rput(3.75,1.7){$_-$}
      \rput(4.75,0.7){$_-$}
      \psline(3,2.5)(3.375,2.5)
      \psline(3.375,2)(3.375,2.5)
      \rput(3.625,2.75){$\swarrow$}
      \rput(4.35,3.25){$\npo$}
      \rput(1,1.5){$\lambda$}
      \endpspicture
      \caption{\label{fig:monotonic} Identifying the unique position
        for $\npo$ based on $\sigma_n$.}
    \end{center}
  \end{figure}

  Let $\C$ and $\mathcal{D}$ be two distinct connected components of
  the $(\nmo,\npo)$-restriction of $\G$ such that there exist vertices
  $v$ of $\C$ and $u$ of $\mathcal{D}$ with $\{v,u\} \in
  E_{\nmo}$. Let $\phi(\C) \cong \G_{\mu}$, and let $\phi(\mathcal{D})
  \cong \G_{\nu}$. Since $\{v,u\} \in E_{\nmo}$, $\phi(v)$ must have
  $\nmo$ in position $\lambda/\mu$ with $\nmt$ lying between $\nmo$
  and $n$ in the content reading word. Since $\phi$ preserves
  $E_{\nmo}$ edges, $\phi(u)$ must be the result of an elementary dual
  equivalence on $\phi(v)$ for $\nmt, \nmo, n$, which will necessarily
  interchange $\nmo$ and $n$. Since $\phi$ preserves signatures,
  $\lambda/\nu$ lies northwest of the position of $\lambda/\mu$ if and
  only if $\sigma(v)_{\nmt,\nmo} = + -$ and $\sigma(u)_{\nmt,\nmo} = -
  +$. If $\lambda/\nu$ lies northwest of the position of $\lambda/\mu$
  and $\sigma(v)_{n} = -1$, then that $\sigma(v)_{n} =
  \sigma(v)_{\nmo}$. Thus, by axiom 3, $\sigma(u)_{n} = \sigma(v)_{n}
  = -1$. Similarly, if $\lambda/\nu$ lies northwest of the position of
  $\lambda/\mu$ and $\sigma(u)_{n} = +1$, then $\sigma(u)_{n} =
  \sigma(u)_{\nmo}$. Thus, by axiom 3, $\sigma(v)_{n} = \sigma(u)_{n}
  = +1$.

  Abusing notation and terminology, we identify $\phi(\C)$ with the
  cell of $\lambda$ in position $\lambda/\mu$, where $\phi(\C) \cong
  \G_{\mu}$. With this convention, we have shown that if $\sigma_n(\C)
  = +1$ and $\mathcal{D}$ is any component connected to $\C$ by an
  $\nmo$-edge such that $\phi(\mathcal{D})$ lies northwest of
  $\phi(\C)$, then $\sigma_n(\mathcal{D}) = +1$ as well. Similarly, if
  $\sigma_n(\C) = -1$ and $\mathcal{D}$ is any component connected to
  $\C$ by an $\nmo$-edge such that $\phi(\mathcal{D})$ lies southeast
  of $\phi(\C)$, then $\sigma_n(\mathcal{D}) = -1$ as well. By dual
  equivalence graph axiom 6, for any two distinct connected components
  $\C$ and $\mathcal{D}$ of the $(\nmo,\npo)$-restriction of $\G$ and
  any pair of vertices $w$ on $\C$ and $x$ on $\mathcal{D}$, there is
  a path from $w$ to $x$ crossing at most one, and hence exactly one,
  $\nmo$ edge. Therefore for any $\C$ and $\mathcal{D}$, there exist
  vertices $v$ of $\C$ and $u$ of $\mathcal{D}$ such that $\{v,u\} \in
  E_{\nmo}$. Hence every two connected components of the
  $(\nmo,\npo)$-restriction of $\G$ are connected by an $\nmo$-edge,
  thus establishing the monotonicity depicted in
  Figure~\ref{fig:monotonic}.

  This established, it follows that there exists a unique row such
  that $\sigma(\C)_n = -1$ whenever the $\phi(\C)$ has $n$ south of
  this row and $\sigma(\C)_n = +1$ whenever the $\phi(\C)$ has $n$
  north of this row. In this case, the cell containing $\npo$ must be
  placed at the eastern end of this pivotal row, and doing so extends
  $\phi$ to an isomorphism between $(n,\npo)$ graphs.
\end{proof}

Once Theorem~\ref{thm:isomorphic} has been proved,
Lemma~\ref{lem:extend-signs} may be used to obtain the
following generalization of Theorem~\ref{thm:isomorphic} for dual
equivalence graphs of type $(n,N)$.

\begin{corollary}
  Every connected component of a dual equivalence graph of type
  $(n,N)$ is isomorphic to $\G_{\lambda,A}$ for a unique partition
  $\lambda$ and some skew tableau $A$ of shape $\rho/\lambda$, $|\rho|
  = N$, with entries $\npo, \ldots, N$.
\end{corollary}

Finally we have all of the ingredients necessary to prove the main
result of this section.

\begin{theorem}
  Let $\G$ be a connected signed, colored graph of type $(\npo,\npo)$
  satisfying axioms $1$ through $5$ such that each connected component
  of the $(n,n)$-restriction of $\G$ is isomorphic to a standard dual
  equivalence graph. Then there exists a morphism $\phi$ from $\G$ to
  $\G_{\lambda}$ for some unique partition $\lambda$ of
  $\npo$. 
\label{thm:cover}
\end{theorem}

\begin{proof}
  When $\npo=2$ or, more generally, when $\G$ has no $n$-edges, the
  result follows immediately from
  Lemma~\ref{lem:extend-signs}. Therefore we proceed by induction,
  assuming that $\G$ has at least one $n$-edge and assuming the result
  for graphs of type $(n,n)$.

  By induction, for every connected component $\C$ of the
  $(n,\npo)$-restriction of $\G$, we have an isomorphism from the
  $(n,n)$-restriction of $\C$ to $\G_{\mu}$ for a unique partition
  $\mu$ of $n$. By Lemma~\ref{lem:extend-signs}, this isomorphism
  extends to an isomorphism from $\C$ to $\G_{\mu,A}$ for a unique
  augmenting tableau $A$, say with shape $\lambda/\mu$. We will show
  that for any $\C$ the shape of $\mu$ augmented with $A$ is the
  same and that we may glue these isomorphisms together to obtain a
  morphism from $\G$ to $\G_{\lambda}$.

  \begin{figure}[ht]
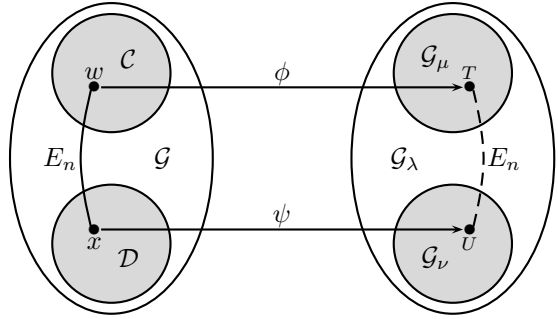

    \begin{center}
      \psset{xunit=3ex}
      \psset{yunit=2.5ex}
      \pspicture(0,0)(14,10)
      \pscircle[fillstyle=solid,fillcolor=lightgray](12,2){.8}
      \rput(12.5,2.5){$\bullet$}
      \rput(12.5,2){$_U$}
      \rput(11.5,1.5){$\G_{\nu}$}
      \pscircle[fillstyle=solid,fillcolor=lightgray](12,8){.8}
      \rput(12.5,7.5){$\bullet$}
      \rput(12.5,8){$_T$}
      \rput(11.5,8.5){$\G_{\mu}$}
      \psellipse(12,5)(3,5.5)
      \rput(10.6,5){$\G_{\lambda}$}
      \pscurve[linestyle=dashed](12.6,2.6)(12.8,3.75)(12.9,5)(12.8,6.25)(12.6,7.4)
      \rput(13.5,5){$E_n$}
      \pscircle[fillstyle=solid,fillcolor=lightgray](2,2){.8}
      \rput(1.5,2.5){$\bullet$}
      \rput(1.5,2){$x$}
      \rput(2.5,1.5){$\mathcal{D}$}
      \pscircle[fillstyle=solid,fillcolor=lightgray](2,8){.8}
      \rput(1.5,7.5){$\bullet$}
      \rput(1.5,8){$w$}
      \rput(2.5,8.5){$\C$}
      \psellipse(2,5)(3,5.5)
      \rput(3.5,5){$\G$}
      \pscurve(1.4,2.6)(1.2,3.75)(1.1,5)(1.2,6.25)(1.4,7.4)
      \rput(0.5,5){$E_n$}
      \psline{->}(1.7,7.5)(12.3,7.5)
      \rput(7,8){$\phi$}
      \psline{->}(1.7,2.5)(12.3,2.5)
      \rput(7,3){$\psi$}
      \endpspicture
      \caption{\label{fig:extend} An illustration of the gluing
        process.}
    \end{center}
  \end{figure}

  Suppose $\{w,x\} \in E_{n}$. Let $\C$ (resp. $\mathcal{D}$) denote
  the connected component of the $(n,\npo)$-restriction of $\G$
  containing $w$ (resp. $x$). Let $\phi$ (resp. $\psi$) be the
  isomorphism from $\C$ (resp. $\mathcal{D}$) to $\G_{\mu,A}$
  (resp. $\G_{\nu,B}$), and set $T = \phi(w)$; see
  Figure~\ref{fig:extend}. We will show that $\psi(x) = E_n(T)$, and
  hence if $\mu,A$ has shape $\lambda$, then so does $\nu,B$ and the
  maps $\phi$ and $\psi$ glue together to give an morphism from $\C
  \cup \mathcal{D}$ to $\G_{\lambda}$ that preserves $n$-edges. There
  are two cases to consider, based on the relative positions of
  $\nmo,n,\npo$ in $T$, regarded as a tableau of shape $\lambda$.

  First suppose that $\npo$ lies between $n$ and $\nmo$ in the reading
  word of $T$. We will show that, in this case, $\C =
  \mathcal{D}$. Since $\npo$ lies between $n$ and $\nmo$ in the
  reading word of $T$, both $\nmo$ and $n$ must be northeastern
  corners of $\mu$, and so there is a cell with entry less than $\nmo$
  that also lies between them. By Proposition~\ref{prop:deshape},
  there exists a tableau $T'$ with $\nmo,n,\npo$ in the same positions
  as in $T$, but now with $\nmt$ lying between $n$ and $\nmo$ in the
  reading word of $T'$. Furthermore, since both $T$ and $T'$ lie on
  the $(\nmt,\npo)$-restriction of $\G_{\mu,A}$, there is a path from
  $T$ to $T'$ in $\G_{\mu,A}$ using only edges $E_h$ with $h \leq
  \nmh$. Let $U' = E_n(T')$. Since $\nmt$ lies between $n$ and $\nmo$
  in $U'$, we have $U' = E_{\nmo}(T')$ as well. By axiom 5, all edges
  in the path from $T$ to $T'$ commute with $E_n$, and so the same
  path takes $U = E_n(T)$ to $U'$, and each pair of corresponding
  tableaux on the two paths is connected by an $E_n$ edge; see
  Figure~\ref{fig:double}. We now embark on a diagram chase around
  Figure~\ref{fig:double}, from $w$ to $x$ and from $T$ to $U$, to
  show that $\C = \mathcal{D}$.

  \begin{figure}[ht]
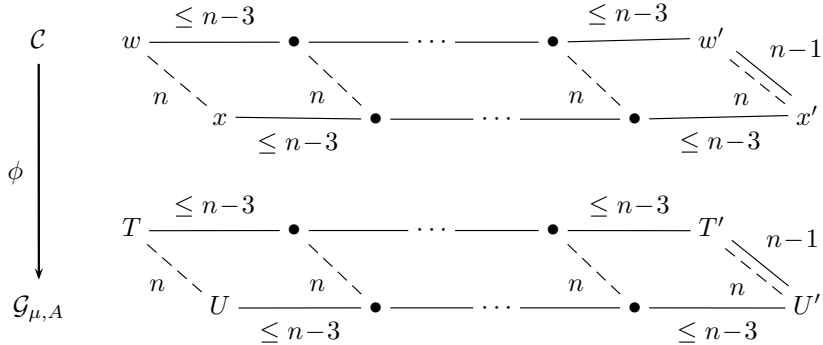

    \begin{center}
      \begin{displaymath}
        \begin{array}{\cs{5}\cs{6}\cs{5}\cs{6}\cs{3}\cs{3}\cs{3}\cs{6}\cs{5}\cs{6}c}
          \rnode{C}{\C} & \rnode{w}{w} & & \rnode{wl}{\B} & & \rnode{wm}{\cdots} 
          & & \rnode{wr}{\B} & & \rnode{wd}{w'} & \\[4ex]
          & & \Rnode{x}{x} & & \rnode{xl}{\B} & & \rnode{xm}{\cdots} 
          & & \rnode{xr}{\B} & & \rnode{xd}{x'} \\[7ex]
          & \Rnode[vref=0.5ex]{T}{T} & & \rnode{Tl}{\B} & & \rnode{Tm}{\cdots} 
          & & \rnode{Tr}{\B} & & \Rnode[vref=0.5ex]{Td}{T'} & \\[4ex]
          \rnode{G}{\G_{\mu,A}} & & \Rnode[vref=0.5ex]{U}{U} & & \rnode{Ul}{\B} 
          & & \rnode{Um}{\cdots} & & \rnode{Ur}{\B} & & \Rnode[vref=0.5ex]{Ud}{U'} 
        \end{array}
        \psset{nodesep=3pt,linewidth=.1ex}
        \ncline {w}{wl}   \naput{\leq \nmh}
        \ncline {wl}{wm}  
        \ncline {wm}{wr}  
        \ncline {wr}{wd}  \naput{\leq \nmh}
        \ncline {x}{xl}   \nbput{\leq \nmh}
        \ncline {xl}{xm}  
        \ncline {xm}{xr}  
        \ncline {xr}{xd}  \nbput{\leq \nmh}
        \ncline[linestyle=dashed] {w}{x} \nbput{n}
        \ncline[linestyle=dashed] {wl}{xl} \nbput{n}
        \ncline[linestyle=dashed] {wr}{xr} \nbput{n}
        \ncline[offset=2pt] {wd}{xd} \naput {\nmo}
        \ncline[linestyle=dashed,offset=2pt] {xd}{wd} \naput{n}
        \ncline[linewidth=.2ex,nodesep=6pt]{->} {C}{G}   \nbput{\phi}
        \ncline {T}{Tl}   \naput{\leq \nmh}
        \ncline {Tl}{Tm}  
        \ncline {Tm}{Tr}  
        \ncline {Tr}{Td}  \naput{\leq \nmh}
        \ncline {U}{Ul}   \nbput{\leq \nmh}
        \ncline {Ul}{Um}  
        \ncline {Um}{Ur}  
        \ncline {Ur}{Ud}  \nbput{\leq \nmh}
        \ncline[linestyle=dashed] {T}{U} \nbput{n}
        \ncline[linestyle=dashed] {Tl}{Ul} \nbput{n}
        \ncline[linestyle=dashed] {Tr}{Ur} \nbput{n}
        \ncline[offset=2pt] {Td}{Ud} \naput {\nmo}
        \ncline[linestyle=dashed,offset=2pt] {Ud}{Td} \naput{n}
      \end{displaymath}
    \end{center}
    \caption{\label{fig:double} Illustration of the path from $T$ to
      $U$ in $\G_{\mu,A}$ and its lift in $\C$.}
  \end{figure}

  Since the path from $T$ to $T'$ to $U'$ to $U$ uses only edges from
  $\G_{\mu,A}$, this path lifts via the isomorphism $\phi$ to a path
  in $\C$. Let $w' = \phi^{-1}(T')$ and $x' = \phi^{-1}(U')$. We will
  show that $x = \phi^{-1}(U)$ and so lies on $\C$. Since $\phi$
  preserves signatures, both $w'$ and $x'$ must admit an $n$-edge in
  $\G$. As summarized in Figure~\ref{fig:lambda4}, axioms 3 and 4
  dictate that the only way for two vertices connected by an
  $\nmo$-edge both to admit an $n$-edge is for $\{w',x'\} \in E_n$ in
  $\G$. By axioms 2 and 5, the path from $w'$ to $w$ gives an
  identical path from $x'$ to $\phi^{-1}(U)$. Since each corresponding
  pair along the two paths must be paired by an $n$-edge, we have
  $\phi^{-1}(U) = E_n(w) = x$, as desired.  Therefore $x$ lies on
  $\C$, and $\phi$ respects the $n$-edge between $w$ and $x$. In this
  case $\C=\mathcal{D}$ and, by Proposition~\ref{prop:noauto-noniso},
  $\psi=\phi$.

  For the second case, suppose that $\nmo$ lies between $n$ and $\npo$
  in $T$. Consider the subset of tableaux in $\G_{\mu,A}$ with $n$ and
  $\npo$ fixed in the same position as in $T$ and $\nmo$ lying
  anywhere between them in the reading word. We claim that this set
  uniquely determines $\mu$ and $A$. In terms of the graph structure,
  these are all tableaux reachable from $T$ using edges $E_h$ with $h
  \leq \nmh$ and a certain subset of the $E_{\nmt}$ edges. We will
  return soon to the question of which $E_{\nmt}$ edges these are. For
  now, let $\mathcal{T}$ denote the union of the graphs $\G_{\rho,R}$,
  where $\rho$ is a partition of $\nmt$ with augmenting tableau $R$
  consisting of a single cell containing $\nmo$ such that $\rho,R$ is
  the shape of $T$ with $n$ and $\npo$ removed and the augmented cell
  of $R$ lies strictly between the positions of $n$ and $\npo$ in
  $T$. Clearly the set of $\rho,R$ uniquely determines the cells
  containing $n$ and $\npo$, and so uniquely determines
  $\lambda$. Furthermore, which of $n,\npo$ occupies which cell is
  determined by $\sigma_n$, and this is constant on this subset by
  axioms $2$ and $3$. Lifting $\mathcal{T}$ to $\C$ using $\phi^{-1}$
  gives rise to an induced subgraph of $\C$ that completely determines
  $\lambda$ as well as the positions of $n$ and $\npo$ in the image of
  this subgraph under $\phi$. We will show that the corresponding
  induced subgraph for $\mathcal{D}$ is isomorphic but with the
  opposite sign for $\sigma_n$.
  
  \begin{figure}[ht]
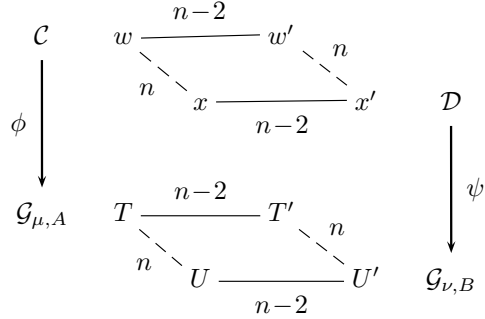

    \begin{center}
      \begin{displaymath}
        \begin{array}{\cs{4}\cs{5}\cs{5}\cs{5}\cs{4}c}
          \rnode{C}{\C} & \rnode{w}{w} & & \rnode{wd}{w'} & & \\[3ex]
          & & \Rnode{x}{x} & & \rnode{xd}{x'} & \rnode{D}{\mathcal{D}} \\[7ex]
          \rnode{A}{\G_{\mu,A}} & \Rnode[vref=0.5ex]{T}{T} & & \Rnode[vref=0.5ex]{Td}{T'} & & \\[3ex]
          & & \Rnode[vref=0.5ex]{U}{U} & & \Rnode[vref=0.5ex]{Ud}{U'} & \rnode{B}{\G_{\nu,B}} 
        \end{array}
        \psset{nodesep=3pt,linewidth=.1ex}
        \ncline {w}{wd}   \naput{\nmt}
        \ncline {x}{xd}   \nbput{\nmt}
        \ncline {T}{Td}   \naput{\nmt}
        \ncline {U}{Ud}   \nbput{\nmt}
        \ncline[linestyle=dashed] {w}{x} \nbput{n}
        \ncline[linestyle=dashed] {wd}{xd} \naput{n}
        \ncline[linestyle=dashed] {T}{U} \nbput{n}
        \ncline[linestyle=dashed] {Td}{Ud} \naput{n}
        \ncline[linewidth=.2ex,nodesep=6pt]{->} {C}{A}   \nbput{\phi}
        \ncline[linewidth=.2ex,nodesep=6pt]{->} {D}{B}   \naput{\psi}
      \end{displaymath}
    \end{center}
    \caption{\label{fig:typeC} Illustration of $E_{\nmt}$ edges on
      $\mathcal{T} \cup \mathcal{U}$ and their lifts in $\C \cup
      \mathcal{D}$.}
  \end{figure}

  To prove the assertion, we return to the question of which
  $E_{\nmt}$ edges are allowed in generating $\mathcal{T}$. Any
  $E_{\nmt}$ edge that keeps $\nmo$ between $n$ and $\npo$ clearly
  does not change $\sigma_{\nmo}$ or $\sigma_{n}$. Therefore such
  $E_{\nmt}$ edges must pair vertices both of which admit an
  $n$-neighbor. Further, neither of these vertices may have $E_n$ as a
  double edge with $E_{\nmo}$ since $\nmo$ lies between $n$ and
  $\npo$. By axiom 4, the $E_{\nmt}$ edges that meet these conditions
  are precisely those in the lower component of
  Figure~\ref{fig:lambda5}. In particular, these $E_{\nmt}$ edges
  commute with $E_n$ edges as depicted in Figure~\ref{fig:typeC}. By
  axiom 5, $E_h$ also commutes with $E_n$ for $h \leq \nmh$. Therefore
  all edges on the induced subgraph of $\C$ containing
  $\phi^{-1}(\mathcal{T})$ commute with $E_n$. Therefore $E_n$ may be
  regarded as an isomorphism from this subgraph to $\mathcal{X} =
  E_n(\phi^{-1}(\mathcal{T}))$. Since $\{w,x\} \in E_n$ and $w \in
  \phi^{-1}(\mathcal{T})$, we have $x \in \mathcal{X}$. Since all
  edges of the induced subgraph have color at most $\nmt$, it follows
  that $\mathcal{X} \subset \mathcal{D}$.

  Let $U = \psi(x)$, and let $\mathcal{U} = \psi(\mathcal{X})$. Since
  $\phi,\psi$ and $E_n$ are isomorphisms of the
  $(n-2,n-2)$-restrictions, $\mathcal{U}$ together with its induced
  edges is isomorphic to $\mathcal{T}$ together with its induced
  edges, though, by axiom 1, the signs for $\sigma_n$ and
  $\sigma_{\npo}$ are reversed. By the definition of $\mathcal{T}$ and
  the fact that it uniquely determines $\mu$ and $A$, this implies
  that the tableaux in $\mathcal{U}$ have shape $\lambda$, with the
  cells containing $n$ and $\npo$ reversed from that in
  $\mathcal{T}$. In particular, $\mathcal{U} = E_n(\mathcal{T})$, that
  is to say, $\phi$ and $\psi$ glue to give a morphism from $\C \cup
  \mathcal{D} \subset \G$ to $\G_{\mu,A} \cup \G_{\nu,B} \subset
  \G_{\lambda}$ that respects $E_n$ edges of the induced subgraphs.

  Since $T$ admits an $n$-neighbor, $n$ cannot lie between $\nmo$ and
  $\npo$, so these two are the only cases. Thus we now have a
  well-defined morphism from the $(n,\npo)$-restriction of $\G$ to the
  $(n,\npo)$-restriction of $\G_{\lambda}$ that respects $n$-edges. As
  such, this map lifts to a morphism from $\G$ to $\G_{\lambda}$.
\end{proof}

By Lemma~\ref{lem:onto}, the morphism of Theorem~\ref{thm:cover} is
necessarily surjective, though in general it need not be injective.
The smallest example where injectivity fails was first observed by
Gregg Musiker in a graph of type $(6,6)$ with generating function $2
s_{(3,2,1)}(X)$; see Figure~\ref{fig:gregg} in
Appendix~\ref{app:axiom6}. 

For any $\G$ satisfying the hypotheses of Theorem~\ref{thm:cover}, the
fiber over each vertex of $\G_{\lambda}$ in the morphism from $\G$ to
$\G_{\lambda}$ has the same cardinality. Letting $\phi$ be the
morphism from $\G$ to $\G_{\lambda}$, there is a bijective
correspondence between connected components of $\phi^{-1}(\G_{\mu})$
and connected components of $\phi^{-1}(\G_{\nu})$ that follows from
the following result.

\begin{corollary}
  Let $\G$ satisfy the hypotheses of Theorem~\ref{thm:cover}, and let
  $\phi$ be the morphism from $\G$ to $\G_{\lambda}$. For any
  connected component $\C$ of the $(n,n)$-restriction of $\G$, say
  with $\phi(\C) = \G_{\mu}$, and any partition $\nu \subset \lambda$
  of size $n$, there is a unique connected component $\mathcal{D}$ of
  the $(n,n)$-restriction of $\G$ with $\phi(\mathcal{D}) = \G_{\nu}$
  that can be reached from $\C$ by crossing at most one $E_n$ edge.
  \label{cor:fibers}
\end{corollary}

\begin{proof}
  To prove existence, if $\nu \neq \mu$, let $T$ be a tableau of shape
  $\lambda$ with $\npo$ in position $\lambda/\mu$, $n$ in position
  $\lambda/\nu$, and $\nmo$ lying between in the reading
  word. Otherwise let $T$ be a tableau with $\npo$ in position
  $\lambda/\mu$ and $n$ and $\nmo$ lying on opposite sides in the
  reading word. Let $w$ be the unique element in $\phi^{-1}(T) \cap
  \C$. Then $w$ admits an $n$-neighbor, and, since $\phi$ is a
  morphism, $\phi(E_n(w)) = E_n(\phi(w)) \in \G_{\nu}$.

  To prove uniqueness, let $\{w,x\} \in E_n$ with $w \in \C \cong
  \G_{\mu}$ and $x \in \mathcal{D} \cong \G_{\nu}$. If $\npo$ lies
  between $n$ and $\nmo$ in $\phi(w)$, then $\mu = \nu$, and just as
  in the proof of Theorem~\ref{thm:cover}, we concluded that
  $\mathcal{D} = \C$ as desired. Alternately, assume $\nmo$ lies
  between $n$ and $\npo$ in $\phi(w)$, and suppose $\{w',x'\} \in
  E_{\nmo}$ with $w' \in \C$ and $x' \in \mathcal{D}' \cong
  G_{\nu}$. Since $\phi(w)$ and $\phi(w')$ have the same shape, and
  $E_n(\phi(w)) = \phi(E_n(w)) = \phi(x)$ and $E_n(\phi(w')) =
  \phi(E_n(w')) = \phi(x')$ have the same shape, just as in the proof
  of Theorem~\ref{thm:cover}, there must be a path from $\phi(w)$ to
  $\phi(w')$ in $\G_{\nu}$ using only edges $E_h$ with $h \leq \nmh$
  and those $E_{\nmt}$ that commute with $E_n$. Therefore this path
  gives rise to the same path from $\phi(x)$ to $\phi(x')$ in
  $\G_{\mu}$. The former path lifts to a path from $w$ to $w'$ in
  $\C$, and so the latter lifts to a path from $E_n(w)=x$ to
  $E_n(w')=x'$ in $\mathcal{D} = \mathcal{D}'$, which is as desired.
\end{proof}

In order to ensure that the morphism in the conclusion of
Theorem~\ref{thm:cover} is an isomorphism, and thereby complete the
proof of Theorem~\ref{thm:isomorphic}, we need only invoke the
heretofore uninvoked axiom 6.

\begin{proof}[Proof of Theorem~\ref{thm:isomorphic}]
  Let $\G$ be a dual equivalence graph of type $(\npo,\npo)$. We aim
  to show that $\G$ is isomorphic to $\G_{\lambda}$ for a unique
  partition $\lambda$ of $\npo$. We proceed by induction on $\npo$,
  noting that the result is trivial for $\npo = 2$.  Every connected
  component of the $(n,n)$-restriction of $\G$ is a dual equivalence
  graph, and so, by induction, is isomorphic to a standard dual
  equivalence graph. Thus, by Theorem~\ref{thm:cover}, there exists a
  morphism, say $\phi$, from $\G$ to $\G_{\lambda}$ for a unique
  partition $\lambda$ of $\npo$. By Corollary~\ref{cor:fibers}, for
  any connected component $\C$ of the $(n,n)$-restriction of $\G$ and
  any partition $\nu \subset \lambda$ of size $n$, there is a unique
  connected component $\mathcal{D}$ of the $(n,n)$-restriction of $\G$
  that can be reached from $\C$ by crossing at most one $E_n$ edge
  such that $\phi(\mathcal{D}) = \G_{\nu}$. By dual equivalence axiom
  $6$, any two connected components of the $(n,n)$-restriction of $\G$
  can be connected by a path using at most one $E_n$ edge. Therefore
  the connected components of the $(n,n)$-restriction of $\G$ are
  pairwise non-isomorphic. Hence the morphism from $\G$ to
  $\G_{\lambda}$ is injective on the $(n,\npo)$-restrictions, and so
  it is injective on all of $\G$. Surjectivity follows from
  Lemma~\ref{lem:onto}, thus $\phi$ is an isomorphism.
\end{proof}

%
\section{Abstract dual equivalence}
%
\label{sec:jdt}

\subsection{A re-characterization of dual equivalence}
\label{sec:jdt-involutions}

The inspiration for dual equivalence graphs comes from the elementary
dual equivalence involutions. In this section, we reformulate the
machinery of dual equivalence graphs back in terms of
involutions. Given a set $\mathcal{A}$ of combinatorial objects
together with a notion of a descent set $\mathrm{Des}$ sending an
object to a subset of positive integers, the goal is to define
equivalence relations on objects in $\mathcal{A}$ so that the sum over
objects in any single equivalence class is a single Schur function.

Given $(\mathcal{A}, \mathrm{Des})$ and involutions $\varphi_2,
\ldots, \varphi_{n-1}$ on $\mathcal{A}$, for $1 < h < i < n$ we
consider the \emph{restricted dual equivalence class}
$[T]_{(h-1,i+1)}$ generated by $\varphi_h,\ldots,\varphi_i$. In
addition, we consider the \emph{restricted and shifted descent set}
$\mathrm{Des}_{(h-1,i+1)}(T)$ obtained by intersecting $\mathrm{Des}(T)$
with $\{h-1,\ldots,i\}$ and subtracting $h-2$ from each element so
that $\mathrm{Des}_{(h-1,i+1)}(T) \subseteq [i-h+2]$.

\begin{definition}
  Let $\mathcal{A}$ be a finite set, and let $\mathrm{Des}$ be a map
  on $\mathcal{A}$ such that $\mathrm{Des}(T) \subseteq [n-1]$ for all
  $T \in \mathcal{A}$. A \emph{dual equivalence for
    $(\mathcal{A},\mathrm{Des})$} is a family of involutions
  $\{\varphi_i\}_{1<i<n}$ on $\mathcal{A}$ such that

  \renewcommand{\theenumi}{\roman{enumi}}
  \begin{enumerate}
  \item For all $i-h \leq 3$ and all $T \in \mathcal{A}$, there
    exists a partition $\lambda$ of $i-h+3$ such that
    \[ \sum_{U \in [T]_{(h,i)}} Q_{\mathrm{Des}_{(h,i)}(U)}(X) = s_{\lambda}(X). \] 
    
  \item For all $|i-j| \geq 3$ and all $T \in\mathcal{A}$, we have
    \begin{displaymath}
      \varphi_{j} \varphi_{i}(T) = \varphi_{i} \varphi_{j}(T).
    \end{displaymath}

  \end{enumerate}

  \label{def:strong}
\end{definition}

On the surface, Definition~\ref{def:strong} appears to have only two
conditions. However, condition (i) includes four cases, for
$i-h=0,1,2,3$. These cases correspond to the dual equivalence graph
axioms as follows: $i-h=0$ means there is one edge color, and this
corresponds to axioms $1$ and $2$; $i-h=1$ means there are two edge
colors, and this corresponds to axiom $3$ and the first half of axiom
$4$; $i-h=2$ means there are three edge colors, and this corresponds
to the second half of axiom $4$; $i-h=3$ means there are four edge
colors, and this corresponds to the axiom $6$.

\begin{theorem}
  For $\{\varphi_i\}$ a family of involutions on $\mathcal{A}$, let
  $\G = (\mathcal{A},\sigma,\Phi)$ be the corresponding signed,
  colored graph with descent signature and $i$-colored edges given by
  \begin{equation}
    \sigma(T)_i = -1 \ \Leftrightarrow \ i \in \mathrm{Des}(T)
    \hspace{3em}
    \Phi_i \ = \ \left\{ \ \{T,\varphi_i(T)\} \ | \ T
    \not\in\mathcal{A}^{\varphi_i} \right\}. 
  \end{equation}
  Then $\{\varphi_i\}$ is a dual equivalence for
  $(\mathcal{A},\mathrm{Des})$ if and only if $\G$ is a dual
  equivalence graph.
  \label{thm:involution_graph}
\end{theorem}

\begin{proof}
  We may assume that $\G$ is connected. If $\G$ is a dual equivalence
  graph, then by Theorem~\ref{thm:isomorphic}, we may assume $\G =
  \G_{\lambda}$ for some partition $\lambda$, and the result follows
  from the observation that $d_i$ is a dual equivalence for
  $\mathrm{SYT}(\lambda)$ \cite{Haiman1992}.

  Assume then that $\{\varphi_i\}$ is a dual equivalence for
  $(\mathcal{A},\mathrm{Des})$. If $\sigma(w)_{\imo} = \sigma(w)_i$,
  then the degree $3$ generating function of $[w]_{i}$ is $s_{(3)}$ or
  $s_{(1^3)}$, so $w$ must be in an $i$-equivalence class of its own,
  i.e. $\varphi_i(w) = w$. On the other hand, if $\sigma(w)_{\imo} =
  -\sigma(w)_i$, then the equivalence class of $w$ must have
  generating function $s_{(2,1)}$, so there exists $x \neq w$ such
  that $\varphi_i(w) = x$. Moreover, since $\varphi_i$ is an
  involution, $x$ is unique. This establishes axiom $1$. This also
  shows that, up to interchanging $w$ and $x$, $\sigma(w)_{\imo,i} =
  +-$ and $\sigma(x)_{\imo,i} = -+$, satisfying the first half of
  axiom $2$.

  Assume $x = \varphi_i(w) \neq w$, $\sigma(w)_{\imt} =
  -\sigma(x)_{\imt}$ and, contrary to axiom $3$, $\sigma(w)_{\imt} =
  \sigma(w)_{\imo}$. By axiom $1$ and the first half of axiom $2$, we
  may assume $\sigma(w)_{\imo,i} = +-$ and $\sigma(x)_{\imo,i} =
  -+$. Therefore $\sigma(w)_{\imt,\imo,i} = ++-$ and
  $\sigma(x)_{\imt,\imo,i} = --+$. By restricting to the
  $(\imo,i)$-equivalence class, condition (i) implies that $w$ must
  lie in a class with generating function $s_{(3,1)}$, and $x$ must
  lie in a class with generating function $s_{(2,1,1)}$, contradicting
  the fact that $w,x$ are in the same restricted class with generating
  function a single Schur function. Therefore axiom $3$ must hold.

  Now consider nontrivial connected components of $\Phi_{i-1} \cup
  \Phi_{i}$. Up to conjugating, the restricted generating function is
  either $s_{(3,1)}$ or $s_{(2,2)}$. In the latter case, there are two
  vertices and neither can be a fixed point for $\varphi_{i-1}$ or
  $\varphi_i$ by axiom $1$. Therefore the structure indeed matches the
  right graph of Figure~\ref{fig:lambda4}. In the former case, there
  are three vertices, say $\sigma(w)_{\imt,\imo,i}=++-$,
  $\sigma(x)_{\imt,\imo,i}=+-+$, and $\sigma(u)_{\imt,\imo,i}=-++$. By
  axiom $1$ and the first half of axiom $2$, we must have
  $\varphi_i(w) = x$, $\varphi_{\imo}(u) = x$, with $w$ fixed for
  $\varphi_{\imo}$ and $u$ fixed for $\varphi_i$. Again, the structure
  indeed matches middle graph of Figure~\ref{fig:lambda4}.

  Next we consider nontrivial connected components of $\Phi_{i-1} \cup
  \Phi_{i} \cup \Phi_{i+1}$. Up to conjugating, the restricted
  generating function is either $s_{(4,1)}$, $s_{(3,2)}$ or
  $s_{(3,1,1)}$. In the former case, by axiom $1$, exactly two
  vertices are not fixed points for $\varphi_j$ for each
  $j=\imo,i,\ipo$, and making the forced pairing results in the
  structure of $\G_{(4,1)}$. In the other cases, there are four
  vertices that are not fixed points for $\varphi_j$ for each
  $j=\imo,i,\ipo$, so taking the first half of axiom $2$ into account,
  there are two possible $j$-edge pairings for each
  $j=\imo,i,\ipo$. For generating function $s_{(3,1,1)}$, taking axiom
  $3$ into account uniquely forces each edge pairing so that the
  resulting graph has the structure of $\G_{(3,1,1)}$. For generating
  function $s_{(3,2)}$, given that two color components must appear in
  Figure~\ref{fig:lambda4}, there are two possibilities: the structure
  of $\G_{(3,2)}$ or a vertex $w$ with $\varphi_{\imo}(w) =
  \varphi_i(w) = \varphi_{\ipo}(w)$, in which case the restricted
  equivalence class has only $2$ elements and not the required
  $5$. Therefore axiom $4$ holds.

  Axiom $5$ is precisely the condition (ii) for dual equivalence.

  For the second half of axiom $2$, assume $|i-j|\geq3$. If $w$ is a
  fixed point for $\varphi_i$, then by axiom $5$,
  $\varphi_i(\varphi_j(w)) = \varphi_j(\varphi_i(w)) = \varphi_j(w)$,
  so $\varphi_j(w)$ is a fixed point for $\varphi_i$ as well. In
  particular, if $x = \varphi_i(w) \neq w$ and $\sigma(w)_j \neq
  \sigma(x)_j$ for some $j < \imt$, then by this observation and axiom
  $1$, $\sigma(w)_j \neq \sigma(x)_j$ for every $j < \imt$. Similarly,
  if $x = \varphi_i(w) \neq w$ and $\sigma(w)_j \neq \sigma(x)_j$ for
  some $j > \ipo$, then $\sigma(w)_j \neq \sigma(x)_j$ for every $j >
  \ipo$. Therefore it suffices to show $\sigma(w)_j \neq \sigma(x)_j$
  for $j = \imh$ and $j = \ipt$. This follows by axioms $1$ and $4$
  and the first half of axiom $2$.

  Finally, since axiom 6 follows from axiom 4 when $n \leq 5$, we
  prove axiom 6 by induction and assume $i = n-1$. Since $\G$ has been
  shown to satisfy axioms $1$ through $5$ and the
  $(n-1,n-1)$-restriction satisfies axiom $6$ by induction,
  Theorem~\ref{thm:isomorphic} ensures that the hypotheses of
  Theorem~\ref{thm:cover} and Corollary~\ref{cor:fibers} are
  met. Therefore the generating function of $\G$ is $k s_{\lambda}$
  for some positive integer $k$. When $n=6$, condition (i) ensures
  $k=1$ and the map to $\G_{\lambda}$ is an isomorphism. In
  particular, we may assume the restriction of $\G$ to edges
  $\Phi_{n-4},\ldots,\Phi_{n-1}$ satisfies dual equivalence graph
  axiom $6$.

  Suppose $T,U,V,W \in \mathcal{A}$, and that $\{T,U\}, \{V,W\} \in
  \Phi_{n-1}$ and $U$ and $V$ lie on the same connected component of
  $\Phi_2 \cup \cdots \cup \Phi_{n-2}$. We will show that there exist
  $T^{\prime},W^{\prime}$ lying on the same connected component of
  $\Phi_2 \cup \cdots \cup \Phi_{n-2}$ as $T,W$, respectively, such
  that there is a path from $T^{\prime}$ to $W^{\prime}$ crossing at
  most one $\Phi_{n-1}$ edge; see Figure~\ref{fig:ax6}. This is enough
  to establish axiom 6.

  \begin{figure}[ht]
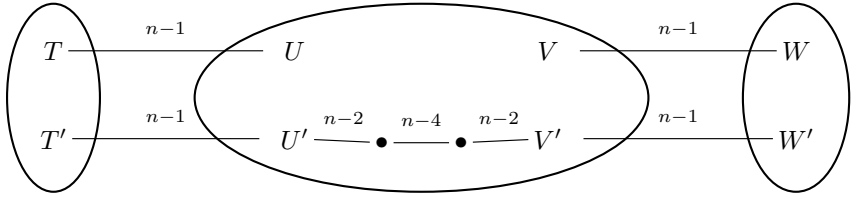

    \begin{displaymath}
      \begin{array}{c@{\hskip 3.5em}c@{\hskip 3.5em}c}
        \ovalnode{L}{ \begin{array}{c}
            \rnode{T}{T} \\ [5ex]
            \rnode{TT}{T^{\prime}}
          \end{array} } &
        \ovalnode{M}{ \begin{array}{c@{\hskip 2.5em}c@{\hskip 2.5em}c@{\hskip 2.5em}c}
            \rnode{U}{U} & & & \rnode{V}{V} \\[5ex]
            \rnode{UU}{U^{\prime}} & \rnode{A}{\bullet} & \rnode{B}{\bullet} & \rnode{VV}{V^{\prime}} 
          \end{array} } &
        \ovalnode{R}{ \begin{array}{c}
            \rnode{W}{W} \\[5ex]
            \rnode{WW}{W^{\prime}} 
          \end{array} }
      \end{array}
      \everypsbox{\scriptstyle}
      \psset{nodesep=2pt,linewidth=.1ex}
      \ncline[nodesepB=8pt] {T}{U}  \naput{n-1}
      \ncline[nodesepB=8pt] {TT}{UU}\naput{n-1}
      \ncline[nodesepA=8pt] {VV}{WW}\naput{n-1}
      \ncline[nodesepA=8pt] {V}{W}  \naput{n-1}
      \ncline {UU}{A} \naput{n-2}
      \ncline {A}{B}  \naput{n-4}
      \ncline {B}{VV} \naput{n-2}
    \end{displaymath}
    \caption{\label{fig:ax6} Establishing axiom 6 from local
      conditions.}
  \end{figure}

  By Theorem~\ref{thm:isomorphic}, every connected dual equivalence
  graph is isomorphic to the graph on tableaux of a given shape. By
  the inductive hypothesis and Theorem~\ref{thm:cover}, we may
  identify $T,U,V,W$ with tableaux of shape $\kappa$, $|\kappa| = n$,
  and, when restricted to entries less than $n$, $T,U,V,W$ have shapes
  $\mu,\lambda,\lambda,\nu$, respectively, with $\mu,\lambda,\nu$
  distinct partitions contained in $\kappa$ by
  Corollary~\ref{cor:fibers}. Then $\kappa\setminus\xi$ must be a
  corner (end of row, top of column) for $\xi =
  \mu,\lambda,\nu$. There are six cases to consider, based on the
  relative positions of $\kappa\setminus\mu$, $\kappa\setminus\nu$,
  and $\kappa\setminus\lambda$ in $\kappa$. We treat one case in full
  detail, noting that the others can be resolved in a completely
  analogous way.

  Assume these cells appear with $\kappa\setminus\mu$ northwest of
  $\kappa\setminus\nu$ northwest of $\kappa\setminus\lambda$. Let
  $T^{\prime}$ be any tableaux of shape $\kappa$ with $n$ in position
  $\kappa\setminus\mu$, $n-1$ in position $\kappa\setminus\lambda$,
  $n-2$ in position $\kappa\setminus\nu$, $n-3$ between $n$ and $n-2$
  in the reading order, $n-4$ between $n-1$ and $n-2$ in the reading
  order, and $n-5$ between $n-3$ and $n-4$ in the reading order. Set
  $U^{\prime} = \varphi_{n-1}(T^{\prime})$. Since the shape of
  $U^{\prime}$ restricted to entries less than $n$ is $\lambda$,
  $U^{\prime}$ must lie on the same connected component of $\Phi_2
  \cup \cdots \cup \Phi_{n-2}$ as $U = \varphi_{n-1}(T)$, as does
  $V^{\prime} = \varphi_{n-2} \varphi_{n-4} \varphi_{n-2}
  (U^{\prime})$. Set $W^{\prime} = \varphi_{n-1}(V^{\prime})$. Since
  the shape of $W^{\prime}$ restricted to entries less than $n$ is
  $\nu$, $W^{\prime}$ must lie on the same connected component of
  $\Phi_2 \cup \cdots \cup \Phi_{n-2}$ as $W =
  \varphi_{n-1}(V)$. Moreover, by unraveling the definitions of
  $U^{\prime}, V^{\prime}$, and $W^{\prime}$, we have $W^{\prime} =
  \varphi_{n-1} \varphi_{n-2} \varphi_{n-4} \varphi_{n-2}
  \varphi_{n-1} (T^{\prime})$. Therefore $T^{\prime}$ and $W^{\prime}$
  lie on the same restricted graph of size $6$, so by induction there
  exists a path from $T^{\prime}$ to $W^{\prime}$ using at most one
  edge in $\Phi_{n-1}$.
\end{proof}

\begin{remark}
  In practice, the characterization of dual equivalence being local
  makes it far better than the axioms for a dual equivalence graph for
  establishing Schur positivity. The equivalence of axiom 6 to a local
  condition was first observed and proved by Roberts
  \cite{Roberts2013}.
\label{rmk:austin}
\end{remark}

\begin{corollary}
  If there exists a dual equivalence for $(\mathcal{A},\mathrm{Des})$,
  then
  \begin{displaymath}
    \sum_{T \in \mathcal{A}} Q_{\mathrm{Des}(T)}(X)
  \end{displaymath}
  is symmetric and Schur positive.
  \label{cor:positivity}
\end{corollary}

\subsection{Formulas from dual equivalence}
\label{sec:jdt-formulas}

Dual equivalence may be used to prove that a function is symmetric and
Schur positive, and it gives a combinatorial interpretation of the
Schur coefficients as the number of equivalence classes or connected
components of a certain type. A classical example of a similar formula
is the standard tableaux version of the Littlewood-Richardson rule.

\begin{theorem}[Littlewood--Richardson Rule]
  Define integers $c_{\mu,\nu}^{\lambda}$ by
  \begin{equation}
    s_{\mu} s_{\nu} = \sum_{\lambda} c_{\mu,\nu}^{\lambda} s_{\lambda}.
  \end{equation}
  Then $c_{\mu,\nu}^{\lambda}$ is the number of standard tableaux of
  shape $\mu$ appended to $\nu$ that rectify by \emph{jeu de taquin}
  to a chosen standard Young tableau of shape $\lambda$.
  \label{thm:LR}
\end{theorem}

We use dual equivalence to give a simple proof of this rule as a
corollary (at $q=1$) to Theorem~\ref{thm:ax1235}. 

Dual equivalence for $\mathcal{A}$ can be regarded as implicitly
giving a rectification rule via the unique isomorphism, say $\theta$,
from $(\mathcal{A},\sigma,\Phi)$ to $\G_{\lambda}$. That is, say that
$T \in \mathcal{A}$ rectifies to $\theta(T) \in
\mathrm{SYT}(\lambda)$. Then choosing only those tableaux that rectify
to a fixed $T\in\mathrm{SYT}(\lambda)$ is equivalent to choosing a
distinguished equivalence class member and taking all objects that map
to that member under $\theta$.

In this section, we present two good choices for distinguished
elements that avoid the need for explicitly constructing the
isomorphism $\theta$.

For $T\in\mathcal{A}$, let $\alpha(T)$ be the composition of $n$
corresponding to $\mathrm{Des}(T)$. Recall the dominance order on
partitions of $n$, which we extend to compositions of $n$ by
\begin{equation}
  \alpha \geq \beta \ \Leftrightarrow \  
  \alpha_1 + \cdots + \alpha_k \geq \beta_1 + \cdots + \beta_k \
  \forall k.
  \label{e:dominance}
\end{equation}
We can now define the set of distinguished elements.

\begin{definition}
  Let $\{\varphi_i\}$ be a dual equivalence for
  $(\mathcal{A},\mathrm{Des})$. Then $T \in \mathcal{A}$ is called \emph{dominant} if $\alpha(T)
  \geq \alpha(S)$ for every $S$ in the dual equivalence class of $T$.
  \label{def:dominant}
\end{definition}

Since dominance order is a partial order, it is not immediately
obvious that dominant objects exist. Not only do they exist, but each
dual equivalence class contains a unique dominant element, and
$\alpha(T)$ is a partition for $T$ dominant.

\begin{theorem}
  Let $\{\varphi_i\}$ be a dual equivalence for $(\mathcal{A},\mathrm{Des})$
  preserving $\mathrm{stat}$. Then
  \begin{equation}
    f(X;q) = \sum_{T \in \mathrm{Dom}(\mathcal{A})} q^{\mathrm{stat}(T)}
    s_{\alpha(T)}(X) = \sum_{\lambda} \left(
      \sum_{\substack{T \in \mathrm{Dom}(\mathcal{A}) \\ \alpha(T)=\lambda}} q^{\mathrm{stat}(T)}
    \right) s_{\lambda}(X),  
    \label{e:schur_f}
  \end{equation}
  where $\mathrm{Dom}(\mathcal{A})$ is the set of dominant objects of $\mathcal{A}$ with respect
  to $\{\varphi_i\}$.
  \label{thm:schur_expansion}
\end{theorem}

\begin{proof}
  Given $\lambda$, let $T_{\lambda} \in \mathrm{SYT}(\lambda)$ denote the
  \emph{superstandard tableau of shape $\lambda$} obtained by filling
  the first row with $1,2,\ldots,\lambda_1$, the second row with
  $\lambda_1+1,\ldots,\lambda_1+\lambda_2$, and so on. For example,
  see Figure~\ref{fig:dominant}. For any $T \in \mathrm{SYT}(\lambda)$, we
  have $\alpha(T) \leq \lambda$ with equality if and only if
  $T=T_{\lambda}$.  Since the dual equivalence classes on tableaux
  include all tableaux of a given shape, each dual equivalence class
  contains a unique dominant element, and the map alpha gives the
  corresponding Schur function for the class. The result for arbitrary
  $\mathcal{A}$ now follows from Corollary~\ref{cor:positivity} since the descent
  sets must be the same for the elements of a dual equivalence class
  and the set of tableaux of shape $\lambda$ for some partition
  $\lambda$.
\end{proof}

\begin{figure}[ht]
  \begin{center}
    \begin{displaymath}
      \tableau{8 & 9 \\ 5 & 6 & 7 \\ 1 & 2 & 3 & 4} 
      \hspace{6em}
      \tableau{3 & 6 \\ 2 & 5 & 8 \\ 1 & 4 & 7 & 9}
    \end{displaymath}
    \caption{\label{fig:dominant}The superstandard (dominant) and
      substandard (subordinate) tableaux of shape $(4,3,2)$.}
  \end{center}    
\end{figure}

\begin{remark}
  Theorem~\ref{thm:schur_expansion} makes use of the implicit
  bijection between $\mathcal{A}$ and tableaux that exists whenever there is a
  dual equivalence for $\mathcal{A}$. This bijection can be realized by
  identifying each $T\in\mathrm{Dom}(\mathcal{A})$ with the superstandard tableau
  $T_{\alpha(T)}$ and then applying the same sequence of dual
  equivalence involutions to both.
  \label{rmk:implicit_bijection}
\end{remark}

\begin{remark}
  There is another distinguished element that can be chosen from each
  equivalence class which is almost as natural as the dominant
  element. Say that $T \in \mathcal{A}$ is \emph{subordinate} if
  $\alpha(T) \leq \alpha(S)$ for every $S$ in the dual equivalence
  class of $T$. For example, the right tableaux in
  Figure~\ref{fig:dominant} is the subordinate tableaux of shape
  $(4,3,2)$. Then each dual equivalence class contains a unique
  subordinate object. Define a map $\beta$ on subordinate objects by
  sending $T$ to $\alpha([n-1] \setminus
  \mathrm{Des}(T))^{\prime}$. That is, complement the set and
  conjugate the shape. Then in Theorem~\ref{thm:schur_expansion}, the
  set $\mathrm{Dom}$ of dominant objects may be replaced with the set
  $\mathrm{Sub}$ of subordinate objects when $\alpha$ is replaced with
  $\beta$.
  \label{rmk:subordinate}
\end{remark}



%
\section{A graph for LLT polynomials}
%
\label{sec:llt}

\subsection{LLT polynomials}
\label{sec:pre-llt}

The LLT polynomial $\widetilde{G}_{\mu}^{(k)}(x;q)$, originally
defined by Lascoux, Leclerc and Thibon \cite{LLT1997} in 1997, is the
$q$-generating function of $k$-ribbon tableaux of shape $\mu$ weighted
by a statistic called cospin. By the Stanton-White correspondence
\cite{StWh1985}, $k$-ribbon tableaux are in bijection with certain
$k$-tuples of tableaux, from which it follows that LLT polynomials are
$q$-analogs of products of Schur functions. An alternative definition
of $\widetilde{G}_{\boldsymbol{\mu}}(x;q)$ as the $q$-generating
function of tuples of semi-standard tableaux of shapes
$\boldsymbol{\mu} = (\mu^{(0)}, \ldots, \mu^{(k-1)})$ weighted by a
statistic called diagonal inversions is given in \cite{HHLRU2005}.
For a detailed account of the equivalence of these definitions
(actually $q^{a}\widetilde{G}_{\mu}^{(k)}(x;q) =
\widetilde{G}_{\boldsymbol{\mu}}(x;q)$ for a constant $a \geq 0$
depending on $\mu$), see \cite{HHLRU2005,Assaf2007}.

Extending prior notation, for $\boldsymbol{\lambda} = (\lambda^{(0)},
\ldots, \lambda^{(k-1)})$, define
\begin{eqnarray*}
  \mathrm{SSYT}(\boldsymbol{\lambda}) & = & 
  \{\mbox{semi-standard tuples of tableaux of shapes
    $(\lambda^{(0)}, \ldots, \lambda^{(k-1)})$} \} , \\  
  \mathrm{SYT}(\boldsymbol{\lambda}) & = & 
  \{\mbox{standard tuples of tableaux of shapes $(\lambda^{(0)},
    \ldots, \lambda^{(k-1)})$} \}.  
\end{eqnarray*}
As with tableaux, if $\mathbf{T}=(T^{(0)}, \ldots, T^{(k-1)}) \in
\mathrm{SSYT}(\boldsymbol{\lambda})$ has entries $1^{\pi_1},
2^{\pi_2}, \ldots$, then we say that $\mathbf{T}$ has \emph{shape}
$\boldsymbol{\lambda}$ and \emph{weight} $\pi$. Note that a standard
tuple of tableaux has weight $(1^n)$, e.g. see
Figure~\ref{fig:LLT-inv}, and this is not the same as a $k$-tuple of
standard tableaux.

\begin{figure}[ht]
  \begin{displaymath}
      \tableau{\\ 7 & 11 \\ 2 & 6 & 10} \hspace{2\cellsize}
      \tableau{\\ 8 \\ 1 & 12}          \hspace{2\cellsize}
      \makebox[0pt]{\rule[-2\cellsize]{.5pt}{\cellsize}}
      \rule[-2\cellsize]{\cellsize}{.5pt}  
      \hspace{2\cellsize}
      \tableau{9 \\ 3 & 5 \\ \cb & 4}
  \end{displaymath}
  \caption{\label{fig:LLT-inv}A standard $4$-tuple of shape $( \ (3,2),
    \ (2,1), \ \varnothing, \ (2,2,1)/(1) \ )$}
\end{figure}

For a $k$-tuple of (skew) shapes $(\lambda^{(0)}, \ldots,
\lambda^{(k-1)})$, define the \emph{shifted content} of a cell $x$ by
\begin{equation}
  \widetilde{c}(x) \; = \; k \cdot c(x) + i
\label{eqn:shifted-content}
\end{equation}
when $x$ is a cell of $\lambda^{(i)}$, where $c(x)$ is the usual
content of $x$ regarded as a cell of $\lambda^{(i)}$. For $\mathbf{T}
\in \mathrm{SSYT}(\boldsymbol{\lambda})$, let $\mathbf{T}(x)$ denote
the entry of the cell $x$ in $\mathbf{T}$. Define the \emph{set of
  diagonal inversions of $\mathbf{T}$} by
\begin{equation}
  \mathrm{dInv}(\mathbf{T}) = \{ (x,y) \; | \; k > \widetilde{c}(y) -
  \widetilde{c}(x) > 0 \; \mbox{and} \; \mathbf{T}(x) > \mathbf{T}(y) \}.
\label{eqn:Invk-T}
\end{equation}
Then the \emph{diagonal inversion number of $\mathbf{T}$} is given by
\begin{equation}
  \mathrm{dinv}(\mathbf{T}) = \left| \mathrm{dInv}(\mathbf{T}) \right| .
\label{eqn:invk-T}
\end{equation}
It will also be convenient to track the \emph{set of diagonal descents of
  $\mathbf{T}$}, defined by
\begin{equation}
  \mathrm{dDes}(\mathbf{T}) = \{ (x,y) \; | \; \widetilde{c}(y) -
  \widetilde{c}(x) = k \; \mbox{and} \; \mathbf{T}(x) > \mathbf{T}(y) \}.
\label{eqn:Desk-T}
\end{equation}
Of course, these descents are determined by the \emph{shape} of
$\mathbf{T}$ by the increasing rows and columns condition for
tableaux.

For example, suppose $\mathbf{T}$ is the $4$-tuple of tableaux in
Figure~\ref{fig:LLT-inv}. Since $\mathbf{T}$ is standard, let us abuse
notation by representing a cell of $\mathbf{T}$ by the entry it
contains. Then the set of diagonal inversions is
\begin{displaymath}
  \mathrm{dInv}(\mathbf{T}) = \left\{ \begin{array}{c}
      (9,7), \ (9,8), \ ( 7,3), \ (8,3), \ (8,2), \ (3,2), \ ( 3,1), \\
      ( 2,1), \ (11,1), \ (11,5), \ ( 6,4), \ (12,4), \ (12,10)
    \end{array} \right\} ,
\end{displaymath}
and so $\mathrm{dinv}(\mathbf{T}) = 13$. The diagonal descents
describe the shape of $\mathbf{T}$. In this case,
\begin{displaymath}
  \mathrm{dDes}(\mathbf{T}) = \left\{ (7,2), (11,6), (8,1), (9,3), (5,4) \right\}.
\end{displaymath}

By \cite{HHLRU2005}, the LLT polynomial
$\widetilde{G}_{\boldsymbol{\mu}}(x;q)$ is given by
\begin{equation}
  \widetilde{G}_{\boldsymbol{\mu}}(x;q) \; = \; 
  \sum_{\mathbf{T} \in \mathrm{SSYT}(\boldsymbol{\mu})}
  q^{\mathrm{dinv}(\mathbf{T})} x^{\mathbf{T}} ,
\label{eqn:llt}
\end{equation}
where $x^{\mathbf{T}}$ is the monomial $x_1^{\pi_1}
x_{2}^{\pi_2}\cdots$ when $\mathbf{T}$ has weight $\pi$.

Notice that when $q=1$, \eqref{eqn:llt} reduces to a product of Schur
functions:
\begin{equation}
  \sum_{\mathbf{T} \in \mathrm{SSYT}(\boldsymbol{\lambda})} x^{\mathbf{T}} \; = \; \prod_{i=0}^{k-1}
  \sum_{T^{(i)} \in  \mathrm{SSYT}(\lambda^{(i)})} x^{T^{(i)}} \; = \;
  \prod_{i=0}^{k-1} s_{\lambda^{(i)}}(x) .
\label{eqn:schurprod} 
\end{equation}

Define the \emph{content reading word} of a tuple of tableaux to be
the word obtained by reading entries in increasing order of shifted
content and reading diagonals southwest to northeast. For the example
in Figure~\ref{fig:LLT-inv}, the content reading word is
$(9,7,8,3,2,11,1,5,6,12,4,10)$.

For $\mathbf{T}$ a standard tuple of tableaux, define
$\sigma(\mathbf{T})$ analogously to \eqref{eqn:sigma} using the
content reading word. Expressed in terms of quasisymmetric functions,
\eqref{eqn:llt} becomes
\begin{equation}
  \widetilde{G}_{\boldsymbol{\mu}}(x;q) \; = \; \sum_{\mathbf{T} \in \mathrm{SYT}(\boldsymbol{\mu})}
  q^{\mathrm{dinv}(\mathbf{T})} Q_{\sigma(\mathbf{T})}(x).
\label{eqn:llt-quasi}
\end{equation}

Using Fock space representations of quantum affine Lie algebras
constructed by Kashiwara, Miwa and Stern \cite{KMS1995}, Lascoux,
Leclerc and Thibon \cite{LLT1997} proved that
$\widetilde{G}_{\boldsymbol{\mu}}(x;q)$ is a symmetric function.  Thus
we may define the Schur coefficients,
$\widetilde{K}_{\lambda,\boldsymbol{\mu}}(q)$, by
\begin{displaymath}
  \widetilde{G}_{\boldsymbol{\mu}}(x;q) = \sum_{\lambda}
  \widetilde{K}_{\lambda,\boldsymbol{\mu}}(q) s_{\lambda}(x) .
\end{displaymath}
Using Kazhdan-Lusztig theory, Leclerc and Thibon \cite{LeTh2000}
proved that $\widetilde{K}_{\lambda,\boldsymbol{\mu}}(q) \in
\mathbb{N}[q]$ for $\boldsymbol{\mu}$ a tuple of straight shapes.  The
original motivation for dual equivalence graphs is to understand these
Schur coefficients combinatorially and for arbitrary skew shapes.

\subsection{Dual equivalence for tuples of tableaux}
\label{sec:llt-edges}

Haiman's dual equivalence involutions do not always preserve the
increasing row and column conditions for a standard tuple of
shapes. Therefore something more is needed.

\begin{definition}
  Define the \emph{elementary twisted dual equivalence involution}
  $\widetilde{d}_i$, $1<i<n$, on permutations $w$ as follows. If $i$
  lies between $i-1$ and $i+1$ in $w$, then
  $\widetilde{d}_i(w)=w$. Otherwise, $\widetilde{d}_i$ cyclically
  rotates $i-1,i,i+1$ so that $i$ lies on the other side of $i-1$ and
  $i+1$.
  \label{defn:ete}
\end{definition}

Note that Haiman's dual equivalence involutions and the twisted dual
equivalence involutions are the only two possible dual equivalences
for $\mathfrak{S}_3$; see Figure~\ref{fig:twisted}. For the general
case, the number of dual equivalences on $\mathfrak{S}_n$ is
determined by the multiplicities of permutations of $n$ with each
possible inverse descent set.

\begin{figure}[ht]
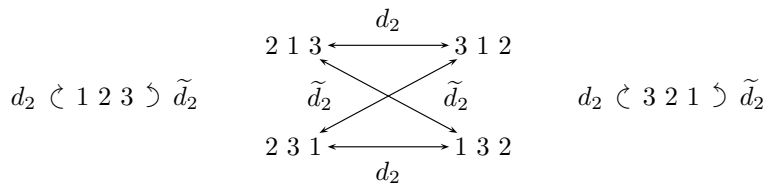

  \begin{center}
    \begin{displaymath}
      d_2 \hspace{1em}
      \raisebox{-.5ex}{\begin{rotate}{90}$\curvearrowright$\end{rotate}}
      \hspace{.5em} 1 \ 2 \ 3 \hspace{.5em} 
      \raisebox{2ex}{\begin{rotate}{-90}$\curvearrowleft$\end{rotate}}
      \hspace{1em} \widetilde{d}_2
      \hspace{2em}\begin{array}{cc}
        \rnode{a}{2 \ 1 \ 3} \hspace{2em} & \hspace{2em}
        \rnode{b}{3 \ 1 \ 2} \\[1.5ex]
        \widetilde{d}_2 & \widetilde{d}_2 \\[1.5ex]
        \rnode{c}{2 \ 3 \ 1} \hspace{2em} & \hspace{2em} \rnode{d}{1 \ 3 \ 2} 
      \end{array} \hspace{2em} 
      d_2 \hspace{1em}
      \raisebox{-.5ex}{\begin{rotate}{90}$\curvearrowright$\end{rotate}}
      \hspace{.5em} 3 \ 2 \ 1 \hspace{.5em} 
      \raisebox{2ex}{\begin{rotate}{-90}$\curvearrowleft$\end{rotate}}
      \hspace{1em} \widetilde{d}_2
      \psset{nodesep=2pt,linewidth=.1ex}
      \ncline{<->}{a}{b} \naput{d_2}
      \ncline{<->}{c}{d} \nbput{d_2}
      \ncline{<->}{a}{d} 
      \ncline{<->}{c}{b} 
    \end{displaymath}
    \caption{\label{fig:twisted}The two dual equivalences for
      $\mathfrak{S}_3$.}
  \end{center}
\end{figure}

Define the distance between two entries $i$ and $j$ of $\mathbf{T} \in
\mathrm{SYT}(\boldsymbol{\mu})$ to be the difference of the shifted
contents of their cells, with the extension
$\mathrm{dist}(a_1,\ldots,a_l) =
\max_{i,j}\{\mathrm{dist}(a_i,a_j)\}$. Note that none of $\triple$ may
occur with the same content. For fixed $k$, combine these two dual
equivalences into an involution $D_i$ on $k$-tuples of tableaux by
\begin{equation}
  D_i(w) \; = \; \left\{
    \begin{array}{ll}
      d_i(w) & \mbox{if} \;\; \mathrm{dist}(\triple) > k \\
      \widetilde{d}_i(w) & \mbox{if} \;\; \mathrm{dist}(\triple) \leq k
    \end{array} \right. 
  \label{eqn:Dk}
\end{equation}
where $w$ is the content reading word of $\mathbf{T}$.

\begin{proposition}
  For $\boldsymbol{\mu}$ a tuple of shapes and $\mathbf{T} \in
  \mathrm{SYT}(\boldsymbol{\mu})$, $\mathbf{T}$ and $D_i(\mathbf{T})$
  have the same diagonal descent set and the same diagonal inversion
  number. In particular, $D_i$ is a well-defined involution on
  $\mathrm{SYT}(\boldsymbol{\mu})$ that preserves the number of
  diagonal inversions.
\label{prop:preserve}
\end{proposition}

\begin{proof}
  If $i$ lies between $\imo$ and $\ipo$ in $w(\mathbf{T})$, then the
  assertion is trivial. Assume then that $i$ does not lie between
  $\imo$ and $\ipo$ in $w(\mathbf{T})$. If $\mathrm{dist}(\triple) >
  k$ in $\mathbf{T}$, then the relative order of the middle letter
  with the two outer letters remains unchanged, and the outer letters
  do not form a potential descent or a potential inversion. Therefore
  $\mathrm{dDes}(\mathbf{T}) = \mathrm{dDes}(d_i(\mathbf{T}))$ and
  $\mathrm{dInv}(\mathbf{T}) = \mathrm{dInv}(d_i(\mathbf{T}))$.

  If $\mathrm{dist}(\triple) \leq k$ in $\mathbf{T}$, then the
  relative order of the outer letters in $\mathbf{T}$ is the same as
  the relative order of the outer letters in
  $\widetilde{d}_i(\mathbf{T})$. If these positions form a potential
  descent, which happens if and only if $\mathrm{dist}(\triple) = k$,
  the descent or lack thereof between the outer two letters is
  preserved. Therefore $\mathrm{dDes}(\mathbf{T}) =
  \mathrm{dDes}(\widetilde{d}_i(\mathbf{T}))$. In the case
  $\mathrm{dist}(\triple) < k$, these outer letters form a potential
  inversion, which similarly is preserved. The middle letter toggles
  under $\widetilde{d}_i$ between the smallest, $i-1$, and the largest
  $i+1$. Therefore exactly one of the potential inversions between
  involving the middle letter is an inversion, and this holds for
  $\mathbf{T}$ as well as for $\widetilde{d}_i(\mathbf{T})$. Therefore
  $\mathrm{dinv}(\mathbf{T}) =
  \mathrm{dinv}(\widetilde{d}_i(\mathbf{T}))$, though
  $\mathrm{dInv}(\mathbf{T}) \neq
  \mathrm{dInv}(\widetilde{d}_i(\mathbf{T}))$.

  Since $D_i$ preserves the diagonal descent set, it is a well-defined
  involution on $\mathrm{SYT}(\boldsymbol{\mu})$. 
\end{proof}

Let $\G_{\boldsymbol{\mu}}$ be the signed, colored graph on
$\mathrm{SYT}(\boldsymbol{\mu})$ determined by the involutions
$D_i$. An example of $\G_{\boldsymbol{\mu}}$ is given in
Figure~\ref{fig:LLT5}.

\begin{figure}[ht]
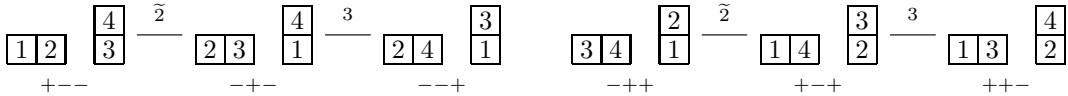

  \begin{displaymath}
    \begin{array}{\cs{4} \cs{4} \cs{4} \cs{4} \cs{4} c}
      \stab{a}{ & & & 4 \\ 1 & 2 & & 3}{+--} &
      \stab{b}{ & & & 4 \\ 2 & 3 & & 1}{-+-} &
      \stab{c}{ & & & 3 \\ 2 & 4 & & 1}{--+} &
      \stab{d}{ & & & 2 \\ 3 & 4 & & 1}{-++} &
      \stab{e}{ & & & 3 \\ 1 & 4 & & 2}{+-+} &
      \stab{f}{ & & & 4 \\ 1 & 3 & & 2}{++-} 
    \end{array}
    \psset{nodesep=5pt,linewidth=.1ex}
    \everypsbox{\scriptstyle}
    \ncline            {a}{b} \naput{\widetilde{2}}
    \ncline            {b}{c} \naput{3}
    \ncline            {d}{e} \naput{\widetilde{2}}
    \ncline            {e}{f} \naput{3}
  \end{displaymath}    
  \caption{\label{fig:LLT5}The graph on domino tableaux of shape $(
    \ (2), \ (1,1) \ )$.}
\end{figure}

Since the graph in Figure~\ref{fig:LLT5} is a dual equivalence graph,
we have
\begin{displaymath}
  \widetilde{G}_{(2),(1,1)}(x;q) = q s_{3,1}(x) + q^2 s_{2,1,1}(x).
\end{displaymath}
In general, $\G_{\boldsymbol{\mu}}$ does not satisfy dual equivalence
axioms $4$ or $6$. Instead of having restricted equivalence classes be
single Schur functions, the restricted equivalence classes for $D_i$
are conjecturally Schur positive.

\begin{definition}
  Let $\mathcal{A}$ be a finite set, and let $\mathrm{Des}$ be a
  descent set map on $\mathcal{A}$ such that $\mathrm{Des}(T)
  \subseteq [n-1]$ for all $T \in \mathcal{A}$.  A \emph{weak dual
    equivalence for $(\mathcal{A},\mathrm{Des})$} is a family of
  involutions $\{\varphi_i\}_{1<i<n}$ on $\mathcal{A}$ such that

  \renewcommand{\theenumi}{\roman{enumi}}
  \begin{enumerate}
  \item For all $|i-j| \leq 3$ and all $T \in \mathcal{A}$,
    \[ \sum_{U \in [T]_{(j,i)}} Q_{\mathrm{Des}_{(j,i)}(U)}(X) \] 
    is Schur positive, and is a single Schur function if $i=j$.

  \item For all $|i-j| \geq 3$ and all $T \in\mathcal{A}$, we have
    \begin{displaymath}
      \varphi_{j} \varphi_{i}(T) = \varphi_{i} \varphi_{j}(T).
    \end{displaymath}

  \end{enumerate}
\label{defn:LSP}
\end{definition}

We refer to condition (i) of Definition~\ref{defn:LSP} as \emph{local
  Schur positivity}.

\begin{theorem}
  For $\boldsymbol{\mu}$ a tuple of shapes, the involutions $\{D_i\}$
  give a weak dual equivalence for $\mathrm{SYT}(\boldsymbol{\mu})$,
  and when $\boldsymbol{\mu}$ consists of at most two shapes, this is
  a strong dual equivalence. 
\label{thm:ax1235}
\end{theorem}

\begin{proof}
  Condition (i) for $i=j$ follows immediately from the fact that both
  $d_i$ and $\widetilde{d}_i$ are dual equivalences for
  $\mathfrak{S}_3$. Condition (ii) follows from the fact that if
  $|i-j| \geq 3$, then $\{\imo,i,\ipo\}$ and $\{j-1,j,j+1\}$ are
  disjoint.

  For $|i-j|=1$, consider first a component of $E_{2} \cup E_{3}$
  containing a vertex with signature $\sigma_{1,2,3} = -++$. The only
  possible permutations with this descent pattern are the three
  depicted on the left side of Figure~\ref{fig:finite}. Applying $d_2$
  and $\widetilde{d}_2$ to each of these gives one of the top four
  permutations in the middle of Figure~\ref{fig:finite}. Applying
  $d_3$ and $\widetilde{d}_3$ to each of these gives either one of the
  three permutations on the right of Figure~\ref{fig:finite}, or the
  bottom permutation in the middle of
  Figure~\ref{fig:finite}. Finally, applying $d_2$ and
  $\widetilde{d}_2$ to $2143$ gives either $1342$ or $3142$, both of
  which appear in the middle of Figure~\ref{fig:finite}. Thus
  traversing the graph in Figure~\ref{fig:finite} by starting on the
  left and alternating between $2$ and $3$ edges must eventually end
  on the right. Taking signatures into account, the possible
  generating functions are of the form $s_{(3,1)} + ms_{(2,2)}$ for
  some $m \in \mathbb{N}$ (in fact, a more detailed analysis shows
  $m=0$ or $1$). In particular, a component containing a vertex with
  signature $\sigma_{1,2,3} = -++$ is Schur positive. The same figure
  applies when working with a vertex with signature $\sigma_{1,2,3} =
  ++-$. Reversing the permutations and multiplying the signatures
  componentwise by $-1$ proves the result for components with a vertex
  with signature $\sigma_{1,2,3} = +--$ or $--+$. A component with a
  vertex with signature $\sigma_{1,2,3} = +++$ or $---$ is a single
  vertex and has generating function $s_{(4)}$ or $s_{(1,1,1,1)}$,
  respectively. The only remaining case is an alternating loop of
  vertices with signatures $\sigma_{1,2,3} = +-+$ and $-+-$. As
  before, the top four permutations in the middle of
  Figure~\ref{fig:finite} must connect by a $2$-edge and a $3$-edge to
  $2143$, thus the component has two vertices and generating function
  $s_{(2,2)}$. Similarly, the reverse of the top four permutations in
  the middle of Figure~\ref{fig:finite} must connect by a $2$-edge and
  a $3$-edge to $3412$, thus the component has two vertices and
  generating function $s_{(2,2)}$. This covers all cases, so local
  Schur positivity holds. Moreover, the case when $\boldsymbol{\mu}$
  has two shapes can be seen to be a single Schur function directly
  from Figure~\ref{fig:2classes}.

  \begin{figure}[ht]
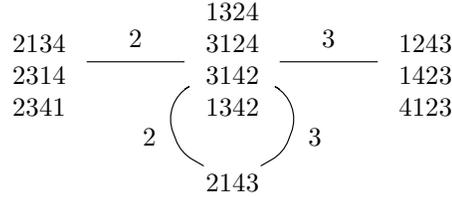

  \begin{displaymath}
    \begin{array}{\cs{10}\cs{10}c}
      \rnode{L}{%
      \begin{array}{c}
        \\
        2134 \\
        2314 \\
        2341
      \end{array}} &
      \rnode{M}{%
      \begin{array}{c}
        1324 \\
        3124 \\
        3142 \\
        1342
      \end{array}} &
      \rnode{R}{%
      \begin{array}{c}
        \\
        1243 \\
        1423 \\
        4123 
      \end{array}} \\[8ex]
      & \rnode{B}{2143} &
    \end{array}
    \psset{nodesep=3pt,linewidth=.1ex}
    \ncline {L}{M} \naput{2}
    \ncline {M}{R} \naput{3}
    \ncdiag[angleA=330,angleB=30,linearc=.5] {M}{B} \naput{3}
    \ncdiag[angleA=150,angleB=210,linearc=.5] {B}{M} \naput{2}
  \end{displaymath}
  \caption{\label{fig:finite}Possible vertices and edges for graph
    using $d_i$ or $\widetilde{d}_i$ for edges colored $i=2,3$.}
  \end{figure}

  For $|i-j|=2,3$, we enumerate all cases and check. Notice that one
  may regard the shape $\boldsymbol{\mu}$ as specifying attacking
  positions in a permutation. That is, for a permutation $w$, say that
  $w_p$ \emph{attacks} $w_q$ if $p<q$ and the difference in shifted
  contents between $p$ and $q$ is at most $k$. Therefore the structure
  of $\G_{\boldsymbol{\mu}}$ is given by the graph on permutations
  where the edges are given by $\widetilde{d}_i$ if $i$ attacks the
  rightmost of $i\pm 1$ or if the leftmost of $i \pm 1$ attacks $i$,
  and by $d_i$ otherwise. Since attacking positions are determined by
  distance, if $w_p$ attacks $w_r$ with $p<q<r$, then $w_p$ attacks
  $w_q$ as well. Therefore the graph on permutations of size $n$ is
  determined by $(a_1,\ldots,a_{n-1})$, where $a_j$ is the rightmost
  position that $w_j$ attacks. Since none of $i-1,i,i+1$ can have the
  same shifted content, we may assume that each position attacks its
  right neighbor, and so $j+1 \leq a_j \leq n$. Moreover, if $w_p$
  attacks $w_r$ with $p<q<r$, then $w_q$ attacks $w_r$ as
  well. Therefore $a_p \leq a_{p+1}$. Hence the number of attacking
  vectors to consider for permutations of length $n$ is the $n-1$st
  Catalan number. In particular, there are $14$ graph structures on
  permutations of $5$ and $42$ on permutations of $6$. These cases can
  be checked by hand or by computer.
\end{proof}

An alternative and completely elementary proof of
Theorem~\ref{thm:ax1235} for $\boldsymbol{\mu}$ consisting of at most
two shapes follows immediately from \cite{A-mahonian08}[Theorem 5.3].

\begin{figure}[ht]
    \begin{displaymath}
      \begin{array}{ccc}
        \{ 2314 \stackrel{\widetilde{d}_2}{\longleftrightarrow} 3124
        \stackrel{d_3}{\longleftrightarrow} 4123 \} &
        \{ 2143 \begin{array}{c}
          \stackrel{d_2}{\longleftrightarrow} \\[-1ex]
          \stackrel{\displaystyle\longleftrightarrow}{_{d_3}}
        \end{array} 3142 \} &
        \{ 1432 \stackrel{d_2}{\longleftrightarrow} 2431
        \stackrel{\widetilde{d}_3}{\longleftrightarrow} 3241 \} \\
        \{ 2341 \stackrel{d_2}{\longleftrightarrow} 1342 
        \stackrel{\widetilde{d}_3}{\longleftrightarrow} 1423 \} & &
        \{ 4312 \stackrel{\widetilde{d}_2}{\longleftrightarrow} 4231
        \stackrel{\widetilde{d}_3}{\longleftrightarrow} 3421 \} \\
        \{ 2134 \stackrel{\widetilde{d}_2}{\longleftrightarrow} 1324
        \stackrel{\widetilde{d}_3}{\longleftrightarrow} 1243 \} &
        \{ 2413 \begin{array}{c}
          \stackrel{d_2}{\longleftrightarrow} \\[-1ex]
          \stackrel{\displaystyle\longleftrightarrow}{_{d_3}}
        \end{array} 3412 \} &
        \{ 4132 \stackrel{\widetilde{d}_2}{\longleftrightarrow} 4213
        \stackrel{d_3}{\longleftrightarrow} 3214 \} 
      \end{array} 
    \end{displaymath}
    \caption{\label{fig:2classes}The nontrivial dual equivalence
      graphs on $\mathfrak{S}_4$ using $\widetilde{d}_i$ only when
      $\imo,i,\ipo$ are adjacent.}
\end{figure}

Theorems~\ref{thm:involution_graph} and \ref{thm:ax1235} together with
Proposition~\ref{prop:preserve} prove the following, which can also be
found in \cite{A-mahonian08}[Corollary 5.4].

\begin{corollary}
  For $\boldsymbol{\mu} = (\mu^{(0)},\mu^{(1)})$, the LLT polynomial
  $\widetilde{G}_{\boldsymbol{\mu}}(x;q)$ is Schur positive.
\end{corollary}

In 1995, Carr\'{e} and Leclerc \cite{CaLe1995} gave a combinatorial
interpretation of $\widetilde{K}_{\lambda,\boldsymbol{\mu}}(q)$ when
$\boldsymbol{\mu}$ has two shapes in their study of $2$-ribbon
tableaux, though a complete proof of their result wasn't found until
2005 by van Leeuwen \cite{vanLeeuwen2005} using the theory of crystal
graphs. That proof is quite long and involved, as compared to the one
page proof above using dual equivalence.

In general, $D_i$ is not a dual equivalence. For instance, if $w$ has
the pattern $2431$ with $\mathrm{dist}(1,2,3) \leq k$ (which forces
$k\geq 3$), then $D_{2}(w)$ contains the pattern $3412$. However,
$D_{3}(w)$ contains the pattern $3241$, which is not the
same. Therefore the restricted generating function is not a single
Schur function (though it is Schur positive). 

\begin{conjecture}
  For $\boldsymbol{\mu}$ a tuple of shapes, each equivalence class
  under $D_i$ is symmetric and Schur positive.
\label{conj:Dgraph}
\end{conjecture}

Conjecture~\ref{conj:Dgraph} has been verified for $n \leq
12$. Additionally, the extreme case when each shape in
$\boldsymbol{\mu}$ is a single (nonskewed) box has connected
components with particularly nice Schur expansions that can be proved
by more elementary means.

In this case, $D_{i} = \widetilde{d}_i$ for all $i$, and so there are
no double edges in $\G_{\boldsymbol{\mu}}$. For the standard dual
equivalence graphs, $\G_{\lambda}$ has no double edges if and only if
$\lambda$ is a hook, i.e. $\lambda = (m,1^{n-m})$ for some
$m \leq n$. Therefore the generating function for a dual equivalence graph
with no double edges is a sum of Schur functions indexed by hooks. The
analog of this fact for $\G_{\boldsymbol{\mu}}$ is that the generating
function is a sum of skew Schur functions indexed by ribbons.

Let $\nu$ be a ribbon of size $n$. Label the cells of $\nu$ from $1$
to $n$ in increasing order of content. Define the \emph{descent set of
  $\nu$}, denoted $\mathrm{Des}(\nu)$, to be the set of indices $i$
such that the cell labeled $\ipo$ lies south of the cell labeled
$i$.  Define the \emph{major index of a ribbon} by
\begin{equation}
  \mathrm{maj}(\nu) = \sum_{i \in \mathrm{Des}(\nu)} i.
\label{eqn:rib-maj}
\end{equation}

Any connected component of $\G_{\boldsymbol{\mu}}$ such that $D_{i} =
\widetilde{d}_i$ on the entire component not only has constant
diagonal inversion number, but the relative ordering of the first and
last letters of each vertex is constant as well. This is because
$\widetilde{d}_i$ does not change the relative order of the outer two
letters among $\triple$. Therefore, $w_1 > w_n$ for some $w \in [u]$
if and only if $w_1 > w_n$ for all $w \in [u]$. In the affirmative
case, say that \emph{$(1,n)$ is an inversion in $[u]$}.

\begin{figure}[ht]
  \begin{center}
    \begin{displaymath}
      \begin{array}{c}
        \{ 2314 \stackrel{\widetilde{d}_2}{\longleftrightarrow} 3124
        \stackrel{\widetilde{d}_3}{\longleftrightarrow} 2143 
        \stackrel{\widetilde{d}_2}{\longleftrightarrow} 1342 
        \stackrel{\widetilde{d}_3}{\longleftrightarrow} 1423 \} 
        \hspace{2em}
        \{ 1432 \stackrel{\widetilde{d}_2}{\longleftrightarrow} 2413
        \stackrel{\widetilde{d}_3}{\longleftrightarrow} 3214 \} \\
        \{ 2341 \stackrel{\widetilde{d}_2}{\longleftrightarrow} 3142 
        \stackrel{\widetilde{d}_3}{\longleftrightarrow} 4123 \}
        \hfill
        \{ 4312 \stackrel{\widetilde{d}_2}{\longleftrightarrow} 4231
        \stackrel{\widetilde{d}_3}{\longleftrightarrow} 3421 \} \\
        \{ 2134 \stackrel{\widetilde{d}_2}{\longleftrightarrow} 1324
        \stackrel{\widetilde{d}_3}{\longleftrightarrow} 1243 \} 
        \hspace{2em}
        \{ 4132 \stackrel{\widetilde{d}_2}{\longleftrightarrow} 4213
        \stackrel{\widetilde{d}_3}{\longleftrightarrow} 3412
        \stackrel{\widetilde{d}_2}{\longleftrightarrow} 2431
        \stackrel{\widetilde{d}_3}{\longleftrightarrow} 3241 \} 
      \end{array} 
    \end{displaymath}
    \caption{\label{fig:ribbon_classes}The twisted dual equivalence
      classes of $\mathfrak{S}_4$.}
  \end{center}
\end{figure}

\begin{theorem}
  Let $[U]$ be an equivalence class for
  $\mathrm{SYT}(\boldsymbol{\mu})$ under $D_i$ for which $D_{i}(T) =
  \widetilde{d}_i(T)$ for all $T \in [U]$. Then
  \begin{equation}
    \sum_{T \in [U]} Q_{\sigma(T)}(x) =  \hspace{-2em}
    \sum_{\substack{\nu \ \mathrm{an} \ n-\mathrm{ribbon} \\ \mathrm{maj}(\nu) =
        \mathrm{dinv}(U) \\ \nmo \in \mathrm{Des}(\nu) \Leftrightarrow
    (1,n) \in \mathrm{dInv}(U)}} \hspace{-2em} s_{\nu}.
  \label{eqn:rib-expand}
  \end{equation}
\label{thm:k=n}
\end{theorem}

\begin{proof}
  We may assume $D_i$ is acting by $\widetilde{d}_i$ on the set of
  permutations of $[n]$, in which case diagonal inversions are the
  usual inversions for permutations. By earlier remarks, for $S,T \in
  [U]$, $\mathrm{inv}(S) = \mathrm{inv}(T)$ and $(1,n) \in
  \mathrm{Inv}(S)$ if and only if $(1,n) \in \mathrm{Inv}(T)$. In
  fact, it is an easy exercise to show by induction that this
  necessary condition for two vertices to coexist in $[U]$ is also
  sufficient. That is to say, $[U]$ is the set of words $S$ with
  $\mathrm{inv}(S) = \mathrm{inv}(U)$ and $(1,n) \in \mathrm{Inv}(S)$
  if and only if $(1,n)$ is an inversion of $U$.

  Recall Foata's bijection on words \cite{Foata1968}. For $w$ a word
  and $x$ a letter, $\phi$ is built recursively using an inner
  function $\gamma_x$ by $ \phi(wx) = \gamma_x \left( \phi(w) \right)
  x$. From this structure it follows that the last letter of $w$ is
  the same as the last letter of $\phi(w)$. Furthermore, $\gamma_x$ is
  defined so that the last letter of $w$ is greater than $x$ if and
  only if the first letter of $\gamma_x(w)$ is greater than $x$, and
  $\phi$ preserves the descent set of the inverse permutation,
  i.e. $\sigma(w) = \sigma(\phi(w))$.  Finally, the bijection
  satisfies $\mathrm{maj}(w) = \mathrm{inv}(\phi(w))$.  Summarizing
  these properties, $\phi$ is a $\sigma$-preserving bijection between
  the following sets:
  \begin{eqnarray*}
    \left\{ w \; | \; \mathrm{inv}(w)=j \; \mbox{and} \; (1,n) \in
      \mathrm{Inv}(w) \right\} & \stackrel{\sim}{\longleftrightarrow}
    & \left\{ w \; | \; \mathrm{maj}(w)=j \; \mbox{and} \; \nmo \in
      \mathrm{Des}(w) \right\}, \\ 
    \left\{ w \; | \; \mathrm{inv}(w)=j \; \mbox{and} \; (1,n) \not\in
      \mathrm{Inv}(w) \right\} & \stackrel{\sim}{\longleftrightarrow}
    & \left\{ w \; | \; \mathrm{maj}(w)=j \; \mbox{and} \; \nmo
      \not\in \mathrm{Des}(w) \right\}. 
  \end{eqnarray*}

  A standard filling of a ribbon $\nu$ is just a permutation $w$ such
  that $\mathrm{Des}(w) = \mathrm{Des}(\nu)$. Therefore by
  \eqref{eqn:quasi-s}, the Schur function $s_{\nu}$ may be expressed
  as
  \begin{equation}
    s_{\nu}(x) = \sum_{\mathrm{Des}(w) = \mathrm{Des}(\nu)} Q_{\sigma(w)}(x).
    \label{eqn:rib-s}
  \end{equation}
  Applying $\phi$ to this formula yields \eqref{eqn:rib-expand}.
\end{proof}

\bibliographystyle{amsalpha} 
\bibliography{../references}

\newcommand{\etalchar}[1]{$^{#1}$}
\providecommand{\bysame}{\leavevmode\hbox to3em{\hrulefill}\thinspace}
\providecommand{\MR}{\relax\ifhmode\unskip\space\fi MR }
\providecommand{\MRhref}[2]{%
  \href{http://www.ams.org/mathscinet-getitem?mr=#1}{#2}
}
\providecommand{\href}[2]{#2}
\begin{thebibliography}{HHL{\etalchar{+}}05}

\bibitem[Ass07]{Assaf2007}
Sami~H. Assaf, \emph{Dual equivalence graphs, ribbon tableaux and {M}acdonald
  polynomials}, Ph.D. thesis, University of California Berkeley, 2007.

\bibitem[Ass08]{A-mahonian08}
Sami Assaf, \emph{A generalized major index statistic}, S\'em. Lothar. Combin.
  \textbf{60} (2008), Art. B50c, 13 pp. (electronic).

\bibitem[CL95]{CaLe1995}
Christophe Carr{\'e} and Bernard Leclerc, \emph{Splitting the square of a
  {S}chur function into its symmetric and antisymmetric parts}, J. Algebraic
  Combin. \textbf{4} (1995), no.~3, 201--231.

\bibitem[Foa68]{Foata1968}
Dominique Foata, \emph{On the {N}etto inversion number of a sequence}, Proc.
  Amer. Math. Soc. \textbf{19} (1968), 236--240.

\bibitem[Ges84]{Gessel1984}
Ira~M. Gessel, \emph{Multipartite {$P$}-partitions and inner products of skew
  {S}chur functions}, Combinatorics and algebra (Boulder, Colo., 1983),
  Contemp. Math., vol.~34, Amer. Math. Soc., Providence, RI, 1984,
  pp.~289--317.

\bibitem[Hai92]{Haiman1992}
Mark~D. Haiman, \emph{Dual equivalence with applications, including a
  conjecture of {P}roctor}, Discrete Math. \textbf{99} (1992), no.~1-3,
  79--113.

\bibitem[HHL{\etalchar{+}}05]{HHLRU2005}
J.~Haglund, M.~Haiman, N.~Loehr, J.~B. Remmel, and A.~Ulyanov, \emph{A
  combinatorial formula for the character of the diagonal coinvariants}, Duke
  Math. J. \textbf{126} (2005), no.~2, 195--232.

\bibitem[KMS95]{KMS1995}
M.~Kashiwara, T.~Miwa, and E.~Stern, \emph{Decomposition of {$q$}-deformed
  {F}ock spaces}, Selecta Math. (N.S.) \textbf{1} (1995), no.~4, 787--805.

\bibitem[LLT97]{LLT1997}
Alain Lascoux, Bernard Leclerc, and Jean-Yves Thibon, \emph{Ribbon tableaux,
  {H}all-{L}ittlewood functions, quantum affine algebras, and unipotent
  varieties}, J. Math. Phys. \textbf{38} (1997), no.~2, 1041--1068.

\bibitem[LT00]{LeTh2000}
Bernard Leclerc and Jean-Yves Thibon, \emph{Littlewood-{R}ichardson
  coefficients and {K}azhdan-{L}usztig polynomials}, Combinatorial methods in
  representation theory (Kyoto, 1998), Adv. Stud. Pure Math., vol.~28,
  Kinokuniya, Tokyo, 2000, pp.~155--220.

\bibitem[Mac95]{Macdonald1995}
I.~G. Macdonald, \emph{Symmetric functions and {H}all polynomials}, second ed.,
  Oxford Mathematical Monographs, The Clarendon Press Oxford University Press,
  New York, 1995, With contributions by A. Zelevinsky, Oxford Science
  Publications.

\bibitem[Rob]{Roberts2013}
Austin Roberts, \emph{Dual equivalence graphs revisited and the explicit
  {S}chur expansion of a family of {LLT} polynomials}, to appear in J.
  Algebraic Combin.

\bibitem[SW85]{StWh1985}
Dennis~W. Stanton and Dennis~E. White, \emph{A {S}chensted algorithm for rim
  hook tableaux}, J. Combin. Theory Ser. A \textbf{40} (1985), no.~2, 211--247.

\bibitem[vL05]{vanLeeuwen2005}
Marc A.~A. van Leeuwen, \emph{Spin-preserving {K}nuth correspondences for
  ribbon tableaux}, Electron. J. Combin. \textbf{12} (2005), Research Paper 10,
  65 pp. (electronic).

\end{thebibliography}

%
\appendix
%

\section{Standard dual equivalence graphs}
\label{app:DEGs}

Below we give the dual equivalence graphs of type $(6,6)$. The graphs
for the conjugate shapes may be obtained by transposing each tableau
and multiplying the signature coordinate-wise by $-1$.


\begin{figure}[ht]
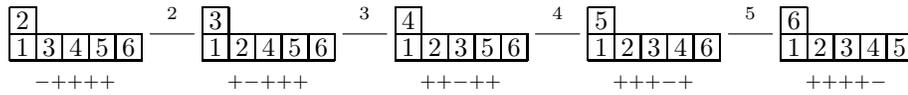

  \begin{displaymath}
    \begin{array}{\cs{3} \cs{3} \cs{3} \cs{3} c}
      \smstab{a}{2 \\ 1 & 3 & 4 & 5 & 6}{-++++} &
      \smstab{b}{3 \\ 1 & 2 & 4 & 5 & 6}{+-+++} &
      \smstab{c}{4 \\ 1 & 2 & 3 & 5 & 6}{++-++} &
      \smstab{d}{5 \\ 1 & 2 & 3 & 4 & 6}{+++-+} &
      \smstab{e}{6 \\ 1 & 2 & 3 & 4 & 5}{++++-} 
    \end{array}
    \psset{nodesep=3pt,linewidth=.1ex}
    \everypsbox{\scriptstyle}
    \ncline            {a}{b} \naput{2}
    \ncline            {b}{c} \naput{3}
    \ncline            {c}{d} \naput{4}
    \ncline            {d}{e} \naput{5}
  \end{displaymath}
\caption{\label{fig:G51}The standard dual equivalence graph $\G_{5,1}$.}
\end{figure}

\begin{figure}[ht]
  \begin{displaymath}
    \begin{array}{\cs{2} \cs{2} \cs{2} \cs{4} \cs{4} \cs{2} \cs{2} c}
      \smstab{a}{3 & 4 \\ 1 & 2 & 5 & 6}{+-+++} & &     
      \smstab{b}{2 & 4 \\ 1 & 3 & 5 & 6}{-+-++} & &     
      \smstab{c}{2 & 5 \\ 1 & 3 & 4 & 6}{-++-+} & & & \\[\cellsize]
      & & & \smstab{i}{2 & 6 \\ 1 & 3 & 4 & 5}{-+++-} & &
      \smstab{j}{3 & 5 \\ 1 & 2 & 4 & 6}{+-+-+} & &
      \smstab{k}{4 & 5 \\ 1 & 2 & 3 & 6}{++-++} \\[\cellsize]
      \smstab{z}{5 & 6 \\ 1 & 2 & 3 & 4}{+++-+} & &     
      \smstab{y}{4 & 6 \\ 1 & 2 & 3 & 5}{++-+-} & &     
      \smstab{x}{3 & 6 \\ 1 & 2 & 4 & 5}{+-++-} & & & 
    \end{array}
    \psset{nodesep=3pt,linewidth=.1ex}
    \everypsbox{\scriptstyle}
    \ncline            {cc}{i} \naput{5}
    \ncline            {ii}{x} \naput{2}
    \ncline            {c}{b} \nbput{4}
    \ncline            {y}{x} \nbput{3}
    \ncline[offset=2pt]{b}{a}\naput{3}
    \ncline[offset=2pt]{a}{b}\naput{2}
    \ncline            {cc}{j} \nbput{2}
    \ncline            {jj}{x} \nbput{5}
    \ncline[offset=2pt]{y}{z}\naput{5}
    \ncline[offset=2pt]{z}{y}\naput{4}
    \ncline[offset=2pt]{j}{k}\naput{3}
    \ncline[offset=2pt]{k}{j}\naput{4}
  \end{displaymath}
\caption{\label{fig:G42}The standard dual equivalence graph $\G_{4,2}$.}
\end{figure}

\begin{figure}[ht]
  \begin{displaymath}
    \begin{array}{\cs{5} \cs{5} \cs{5} c}
      & & \smstab{w}{3 & 4 & 6 \\ 1 & 2 & 5}{+-++-} & \\[.5\cellsize]
      \smstab{z}{4 & 5 & 6 \\ 1 & 2 & 3}{++-++} &
      \smstab{y}{3 & 5 & 6 \\ 1 & 2 & 4}{+-+-+} & &
      \smstab{v}{2 & 4 & 6 \\ 1 & 3 & 5}{-+-+-} \\[.5\cellsize]
      & & \smstab{x}{2 & 5 & 6 \\ 1 & 3 & 4}{-++-+} & 
    \end{array}
    \psset{nodesep=3pt,linewidth=.1ex}
    \everypsbox{\scriptstyle}
    \ncline[offset=2pt]{v}{w} \naput{3}
    \ncline[offset=2pt]{w}{v} \naput{2}
    \ncline            {x}{y} \naput{2}
    \ncline[offset=2pt]{y}{z} \naput{4}
    \ncline[offset=2pt]{z}{y} \naput{3}
    \ncline[offset=2pt]{v}{x} \nbput{4}
    \ncline[offset=2pt]{x}{v} \nbput{5}
    \ncline            {w}{y} \nbput{5}
  \end{displaymath}  
\caption{\label{fig:G33}The standard dual equivalence graph $\G_{3,3}$.}
\end{figure}

\begin{figure}[ht]
  \begin{displaymath}
    \begin{array}{\cs{1} \cs{1} \cs{1} \cs{1} \cs{1} \cs{1} c}
      \smstab{a}{3 \\ 2 \\ 1 & 4 & 5 & 6}{--+++} & &
      \smstab{c}{4 \\ 3 \\ 1 & 2 & 5 & 6}{+--++} & &
      \smstab{f}{5 \\ 4 \\ 1 & 2 & 3 & 6}{++--+} & &
      \smstab{j}{6 \\ 5 \\ 1 & 2 & 3 & 4}{+++--} \\[1.5\cellsize]
      &
      \smstab{b}{4 \\ 2 \\ 1 & 3 & 5 & 6}{-+-++} & &
      \smstab{e}{5 \\ 3 \\ 1 & 2 & 4 & 6}{+-+-+} & &
      \smstab{i}{6 \\ 4 \\ 1 & 2 & 3 & 5}{++-+-} &  \\[1.5\cellsize]
      & & 
      \smstab{d}{5 \\ 2 \\ 1 & 3 & 4 & 6}{-++-+} & &
      \smstab{h}{6 \\ 3 \\ 1 & 2 & 4 & 5}{+-++-} & & \\[1.5\cellsize]
      & & & \smstab{g}{6 \\ 2 \\ 1 & 3 & 4 & 5}{-+++-} & & & 
    \end{array}
    \psset{nodesep=3pt,linewidth=.1ex}
    \everypsbox{\scriptstyle}
    \ncline            {a}{b} \naput{3}
    \ncline            {b}{d} \naput{4}
    \ncline            {c}{e} \naput{4}
    \ncline            {d}{g} \naput{5}
    \ncline            {e}{h} \naput{5}
    \ncline            {f}{i} \naput{5}
    \ncline            {b}{c} \naput{2}
    \ncline            {d}{e} \naput{2}
    \ncline            {e}{f} \naput{3}
    \ncline            {g}{h} \naput{2}
    \ncline            {h}{i} \naput{3}
    \ncline            {i}{j} \naput{4}
  \end{displaymath}  
\caption{\label{fig:G411}The standard dual equivalence graph $\G_{4,1,1}$.}
\end{figure}

\begin{figure}[ht]
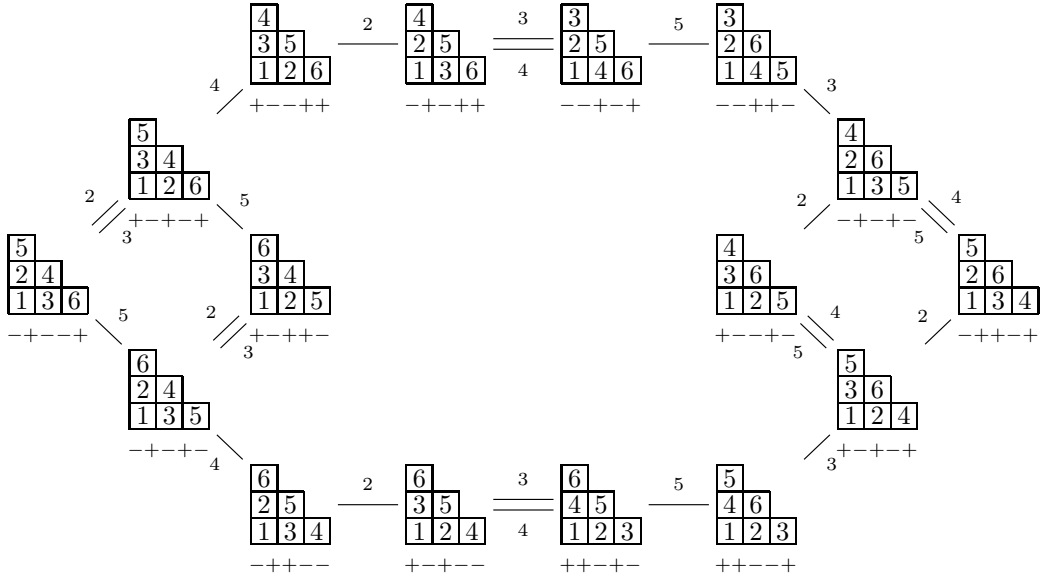

  \begin{displaymath}
    \begin{array}{\cs{1} \cs{1} \cs{4} \cs{4} \cs{4} \cs{1} \cs{1} c}
      & &
      \smstab{b2}{4 \\ 3 & 5 \\ 1 & 2 & 6}{+--++} &
      \smstab{a3}{4 \\ 2 & 5 \\ 1 & 3 & 6}{-+-++} &
      \smstab{a4}{3 \\ 2 & 5 \\ 1 & 4 & 6}{--+-+} &
      \smstab{b5}{3 \\ 2 & 6 \\ 1 & 4 & 5}{--++-} & & \\[1.5\cellsize]
      &
      \smstab{c2}{5 \\ 3 & 4 \\ 1 & 2 & 6}{+-+-+} & & & & &
      \smstab{c5}{4 \\ 2 & 6 \\ 1 & 3 & 5}{-+-+-} & \\[1.5\cellsize]
      \smstab{d1}{5 \\ 2 & 4 \\ 1 & 3 & 6}{-+--+} & &
      \smstab{d3}{6 \\ 3 & 4 \\ 1 & 2 & 5}{+-++-} & & &
      \smstab{d4}{4 \\ 3 & 6 \\ 1 & 2 & 5}{+--+-} & &
      \smstab{d6}{5 \\ 2 & 6 \\ 1 & 3 & 4}{-++-+}  \\ [1.5\cellsize]
      &
      \smstab{e2}{6 \\ 2 & 4 \\ 1 & 3 & 5}{-+-+-} & & & & &
      \smstab{e5}{5 \\ 3 & 6 \\ 1 & 2 & 4}{+-+-+} & \\ [1.5\cellsize]
      & &
      \smstab{f2}{6 \\ 2 & 5 \\ 1 & 3 & 4}{-++--} &
      \smstab{g3}{6 \\ 3 & 5 \\ 1 & 2 & 4}{+-+--} &
      \smstab{g4}{6 \\ 4 & 5 \\ 1 & 2 & 3}{++-+-} &
      \smstab{f5}{5 \\ 4 & 6 \\ 1 & 2 & 3}{++--+} & & 
    \end{array}
    \psset{nodesep=3pt,linewidth=.1ex}
    \everypsbox{\scriptstyle}
    \ncline[offset=2pt]{a3}{a4} \naput{3}
    \ncline[offset=2pt]{a4}{a3} \naput{4}
    \ncline            {b2}{a3} \naput{2}
    \ncline            {a4}{b5} \naput{5}
    \ncline            {b2}{c2} \nbput{4}
    \ncline            {b5}{c5} \naput{3}
    \ncline[offset=2pt]{d1}{c2} \naput{2}
    \ncline[offset=2pt]{c2}{d1} \naput{3}
    \ncline            {c2}{d3} \naput{5}
    \ncline            {d4}{c5} \naput{2}
    \ncline[offset=2pt]{c5}{d6} \naput{4}
    \ncline[offset=2pt]{d6}{c5} \naput{5}
    \ncline            {d1}{e2} \naput{5}
    \ncline[offset=2pt]{e2}{d3} \naput{2}
    \ncline[offset=2pt]{d3}{e2} \naput{3}
    \ncline[offset=2pt]{d4}{e5} \naput{4}
    \ncline[offset=2pt]{e5}{d4} \naput{5}
    \ncline            {e5}{d6} \naput{2}
    \ncline            {e2}{f2} \nbput{4}
    \ncline            {e5}{f5} \naput{3}
    \ncline            {f2}{g3} \naput{2}
    \ncline            {g4}{f5} \naput{5}
    \ncline[offset=2pt]{g3}{g4} \naput{3}
    \ncline[offset=2pt]{g4}{g3} \naput{4}
  \end{displaymath}
  \caption{\label{fig:G321}The standard dual equivalence graph $\G_{3,2,1}$.}
\end{figure}

\clearpage
\section{Necessity of axiom $6$}
\label{app:axiom6}

The example in Figure~\ref{fig:gregg}, first observed by Gregg
Musiker, demonstrates the necessity of axiom $6$. It satisfies axioms
$1$ through $5$, but fails axiom $6$. Comparing with the standard dual
equivalence graphs in Appendix~\ref{app:DEGs}, this graph is a
two-fold cover of $\G_{(3,2,1)}$ as expected from its generating
function $2 s_{(3,2,1)}(X)$. Figure~\ref{fig:tab-gregg} gives the
isomorphism classes of the $(5,6)$-restriction of this graph.

\begin{figure}[ht]
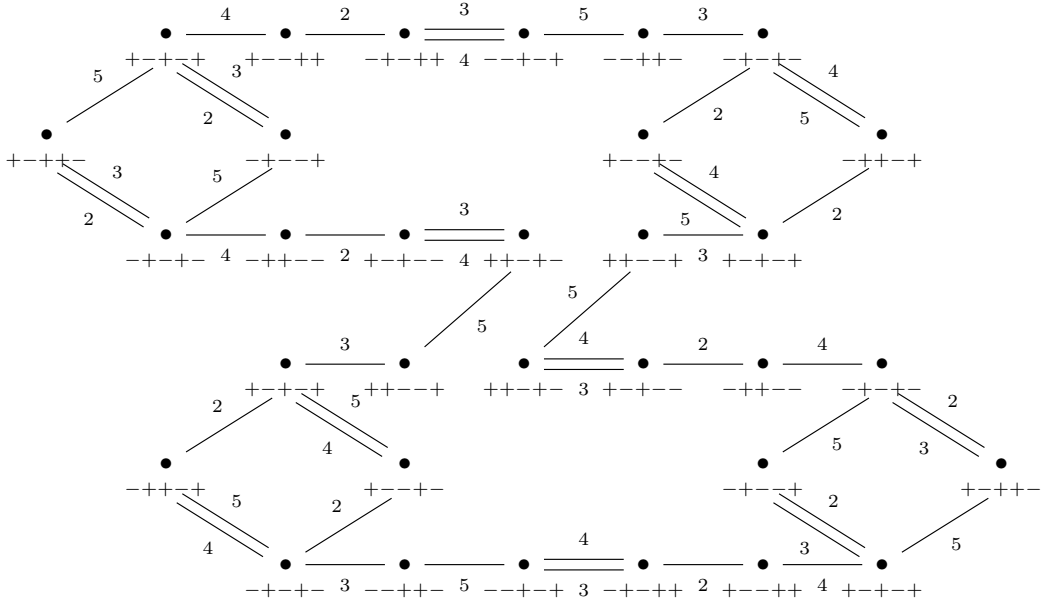

  \begin{displaymath}
    \begin{array}{\cs{7} \cs{7} \cs{7} \cs{7} \cs{7} \cs{7} \cs{7} \cs{7} c}
      &
      \sbull{c2}{+-+-+} &
      \sbull{b2}{+--++} &
      \sbull{a3}{-+-++} &
      \sbull{a4}{--+-+} &
      \sbull{b5}{--++-} &
      \sbull{c5}{-+-+-} &
      & \\[2\cellsize]
      \sbull{d3}{+-++-} & &
      \sbull{d1}{-+--+} & & &
      \sbull{d4}{+--+-} & &
      \sbull{d6}{-++-+} 
      & \\[2\cellsize]
      &
      \sbull{e2}{-+-+-} &
      \sbull{f2}{-++--} &
      \sbull{g3}{+-+--} &
      \sbull{g4}{++-+-} &
      \sbull{f5}{++--+} &
      \sbull{e5}{+-+-+} & 
      & \\[3\cellsize]
      & &
      \sbull{xe5}{+-+-+} & 
      \sbull{xf5}{++--+} &
      \sbull{xg4}{++-+-} &
      \sbull{xg3}{+-+--} &
      \sbull{xf2}{-++--} &
      \sbull{xe2}{-+-+-} & 
      \\[2\cellsize]
      &
      \sbull{xd6}{-++-+} & &
      \sbull{xd4}{+--+-} & & &
      \sbull{xd1}{-+--+}  & &
      \sbull{xd3}{+-++-}
      \\ [2\cellsize]
      & &
      \sbull{xc5}{-+-+-} &
      \sbull{xb5}{--++-} &
      \sbull{xa4}{--+-+} &
      \sbull{xa3}{-+-++} &
      \sbull{xb2}{+--++} &
      \sbull{xc2}{+-+-+} & 
    \end{array}
    \psset{nodesep=5pt,linewidth=.1ex}
    \everypsbox{\scriptstyle}
    \ncline[offset=2pt]{a3}{a4} \naput{3}
    \ncline[offset=2pt]{a4}{a3} \naput{4}
    \ncline            {b2}{a3} \naput{2}
    \ncline            {a4}{b5} \naput{5}
    \ncline            {c2}{b2} \naput{4}
    \ncline            {b5}{c5} \naput{3}
    \ncline[offset=2pt]{d1}{c2c2} \naput{2}
    \ncline[offset=2pt]{c2c2}{d1} \naput{3}
    \ncline            {c2c2}{d3} \nbput{5}
    \ncline            {d4}{c5c5} \nbput{2}
    \ncline[offset=2pt]{c5c5}{d6} \naput{4}
    \ncline[offset=2pt]{d6}{c5c5} \naput{5}
    \ncline            {d1d1}{e2} \nbput{5}
    \ncline[offset=2pt]{e2}{d3d3} \naput{2}
    \ncline[offset=2pt]{d3d3}{e2} \naput{3}
    \ncline[offset=2pt]{d4d4}{e5} \naput{4}
    \ncline[offset=2pt]{e5}{d4d4} \naput{5}
    \ncline            {e5}{d6d6} \nbput{2}
    \ncline            {e2}{f2} \nbput{4}
    \ncline            {f5}{e5} \nbput{3}
    \ncline            {f2}{g3} \nbput{2}
    \ncline            {g4g4}{xf5} \naput{5}
    \ncline[offset=2pt]{g3}{g4} \naput{3}
    \ncline[offset=2pt]{g4}{g3} \naput{4}
    \ncline[offset=2pt]{xa3}{xa4} \naput{3}
    \ncline[offset=2pt]{xa4}{xa3} \naput{4}
    \ncline            {xb2}{xa3} \naput{2}
    \ncline            {xa4}{xb5} \naput{5}
    \ncline            {xc2}{xb2} \naput{4}
    \ncline            {xb5}{xc5} \naput{3}
    \ncline[offset=2pt]{xd1xd1}{xc2} \naput{2}
    \ncline[offset=2pt]{xc2}{xd1xd1} \naput{3}
    \ncline            {xc2}{xd3xd3} \nbput{5}
    \ncline            {xd4xd4}{xc5} \nbput{2}
    \ncline[offset=2pt]{xc5}{xd6xd6} \naput{4}
    \ncline[offset=2pt]{xd6xd6}{xc5} \naput{5}
    \ncline            {xd1}{xe2xe2} \nbput{5}
    \ncline[offset=2pt]{xe2xe2}{xd3} \naput{2}
    \ncline[offset=2pt]{xd3}{xe2xe2} \naput{3}
    \ncline[offset=2pt]{xd4}{xe5xe5} \naput{4}
    \ncline[offset=2pt]{xe5xe5}{xd4} \naput{5}
    \ncline            {xe5xe5}{xd6} \nbput{2}
    \ncline            {xe2}{xf2} \nbput{4}
    \ncline            {xf5}{xe5} \nbput{3}
    \ncline            {xf2}{xg3} \nbput{2}
    \ncline            {xg4}{f5f5} \naput{5}
    \ncline[offset=2pt]{xg3}{xg4} \naput{3}
    \ncline[offset=2pt]{xg4}{xg3} \naput{4}
  \end{displaymath}
  \caption{\label{fig:gregg}The smallest graph satisfying dual
    equivalence graph axioms $1-5$ but not $6$.}
\end{figure}

\begin{figure}[ht]
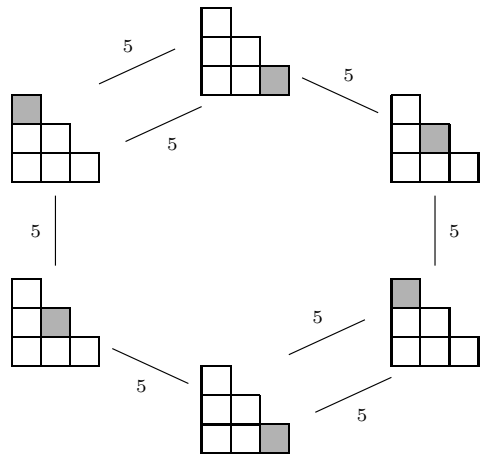

  \begin{displaymath}
    \begin{array}{\cs{9}\cs{9}c}
      & \rnode{a}{\tableau{\e \\ \e & \e \\ \e & \e & \cb}} & \\[\cellsize]
      \rnode{b}{\tableau{\cb \\ \e & \e \\ \e & \e & \e}} & 
      & \rnode{c}{\tableau{\e \\ \e & \cb \\ \e & \e & \e}} \\[5\cellsize]
      \rnode{C}{\tableau{\e \\ \e & \cb \\ \e & \e & \e}} & 
      & \rnode{B}{\tableau{\cb \\ \e & \e \\ \e & \e & \e}} \\[\cellsize]
      & \rnode{A}{\tableau{\e \\ \e & \e \\ \e & \e & \cb}} & 
    \end{array}
    \psset{nodesep=5pt,linewidth=.1ex}
    \everypsbox{\scriptstyle}
    \ncline[offset=12pt] {a}{b} \naput{5}
    \ncline[offset=12pt] {b}{a} \naput{5}
    \ncline {a}{c} \naput{5}
    \ncline {b}{C} \nbput{5}
    \ncline {c}{B} \naput{5}
    \ncline {C}{A} \nbput{5}
    \ncline[offset=12pt] {B}{A} \naput{5}
    \ncline[offset=12pt] {A}{B} \naput{5}
  \end{displaymath}
  \caption{\label{fig:tab-gregg}The $(5,6)$-restriction of
    Figure~\ref{fig:gregg} highlighting the two-fold cover of
    $\G_{(3,2,1)}$.}
\end{figure}

\end{document}